\documentclass[twocolumn]{autart}    

\usepackage{graphics} 
\usepackage{epsfig} 
\usepackage{amsmath} 
\usepackage{amssymb}  
\usepackage{mathtools}
\usepackage[dvipsnames]{xcolor}
\usepackage{empheq}
\usepackage{subcaption}
\usepackage{wrapfig}

\newenvironment{proof}{{\it Proof :~}}{\hfill$\square$\\}
\newcommand{\norm}[1]{\left\lvert#1\right\rvert}

\newcommand{\sigmaU}[1]{\bar{\sigma}(#1)}
\newcommand{\sigmaD}[1]{\underline{\sigma}(#1)}
\newcommand{\dth}{\mathrm{d}\theta}
\newcommand\eps{\varepsilon}

\newcommand\Pn{\mathbf{P}_{\!\!n}}

\newtheorem{theorem}{Theorem}
\newtheorem{definition}{Definition}
\newtheorem{lemma}{Lemma}
\newtheorem{corollary}{Corollary}
\newtheorem{remark}{Remark}
\newtheorem{property}{Property}
\newtheorem{example}{Example}

\makeatletter
\renewcommand*{\@opargbegintheorem}[3]{\trivlist
  \item[\hskip \labelsep{\bfseries #1\ #2}] \textbf{(#3)}\hfill\newline \itshape}
\makeatother

\definecolor{darkred}{rgb}{0.8, 0.25, 0.33}

\begin{document}

\begin{frontmatter}

\title{On the necessity of sufficient LMI conditions for time-delay systems arising from Legendre approximation\thanksref{footnoteinfo}}

\thanks[footnoteinfo]{Corresponding author M.~Bajodek}

\author[LAAS]{Mathieu Bajodek}\ead{mbajodek\@ laas.fr},
\author[LAAS,Seville]{Alexandre Seuret},
\author[LAAS]{Frédéric Gouaisbaut}

\address[LAAS]{LAAS-CNRS, Univ de Toulouse, UPS, 7 avenue du Colonel Roche, 31077 Toulouse, France}
\address[Seville]{Univ. de Sevilla, Camino de los Descubrimientos, s/n 41092 Sevilla, Spain}
\begin{keyword}     
Systems with time-delay; Infinite-dimensional systems; Stability analysis; Lyapunov functionals; LMI; Polynomial approximation.
\end{keyword}               

\begin{abstract}
This work is dedicated to the stability analysis of time-delay systems with a single constant delay using the Lyapunov-Krasovskii theorem. This approach has been widely used in the literature and numerous sufficient conditions of stability have been proposed and expressed as linear matrix inequalities (LMI). The main criticism of the method that is often pointed out is that these LMI conditions are only sufficient, and there is a lack of information regarding the reduction of the conservatism. Recently, scalable methods have been investigated using Bessel-Legendre inequality or orthogonal polynomial-based inequalities. The interest of these methods relies on their hierarchical structure with a guarantee of reduction of the level of conservatism. However, the convergence is still an open question that will be answered for the first time in this paper. The objective is to prove that the stability of a time-delay system implies the feasibility of these scalable LMI, at a sufficiently large order of the Legendre polynomials. Moreover, the proposed contribution is even able to provide an analytic estimation of this order, giving rise to a necessary and sufficient LMI for the stability of time-delay systems. 
\end{abstract}

\end{frontmatter}

\section{Introduction}

Time-delay systems represent a wide class of dynamical systems arising in many applications in electronics, biology, transport, etc...  Their interest in automatic control is natural since they propose many challenging theoretical problems related to their intrinsic infinite-dimensional nature \cite{niculescu2001delay,sipahi2011stability}. In particular, their stability analysis has been at the heart of many research works for several decades~\cite{Fridman2014,Gu2003,Kharitonov2013,Richard2003,zhang2019overview} and many methods have been developed. Among them, the use of Lyapunov-Krasovskii theorem remains one of the most popular techniques because of its inherent robustness.
\newline
In the context of linear time-invariant (LTI) delay systems, several conditions have been obtained over the last two decades by application of the Lyapunov-Krasovskii theorem leading generally to sum of square constraints~\cite{Peet2019,Peet2009,Valmorbida2016} or to linear matrix inequalities (LMI)~\cite{fridman2001stability,he2004delay,xu2008survey}. These conditions, combined with semi-definite convex optimization programs, only led to sufficient conditions. Hence, numerous works aimed at reducing the inherent conservatism and at recovering the necessity using, for instance, discretized functionals~\cite{Gu2013,Gu2003} or delay partitioning methods~\cite{du2009delay,Gouaisbaut2006b}. Among them, an approach based on state extension has been proposed in~\cite{Seuret2015Hierarchy}. The candidate Lyapunov-Krasovskii functional therein depends directly on the projections of the state of the delay system onto Legendre polynomials. This approach led to LMI stability conditions, which benefit from a particular hierarchical structure, arising from the use of the Bessel-Legendre inequality. Similar approaches based on orthogonal polynomials \cite{Lee2019Bessel,park2015auxiliary} have been also considered in the literature. The complexity of these LMI increases with the number of Legendre polynomials $n$ taken under consideration but their conservatism drastically reduces as $n$ increases, at least on examples. However, to the best of our knowledge, the proof of convergence of these LMI to necessary and sufficient LMI conditions of stability is still missing.
\newline
Towards this direction, as for LTI finite-dimensional systems, there exists a converse Lyapunov-Krasovskii theorem for LTI systems with a constant delay. Indeed, it is possible to build a complete Lyapunov-Krasovskii functional for stable LTI delay systems, see for instance \cite{Kharitonov2013} for a larger overview on the problem. While this method has been only seen as a theoretical contributions, the authors of~\cite{Zhabko2015,Zhabko2019} paved the way to use the converse Lyapunov-Krasovskii theorem to derive not sufficient but necessary stability conditions for time-delay systems through approximation. Among them, the authors of \cite{Mondie2014} provided a necessary stability conditions for time-delay systems. This method was then extended to various classes of delay systems, see for instance \cite{Mondie2016,Mondie2018c}. More interestingly, their methods have led to necessary and sufficient stability conditions in \cite{Mondie2017,Mondie2021}. The sufficiency part of their necessary condition has been obtained by estimating an upper bound of the approximation error due to the discretization process of the Lyapunov matrix. It is worth noting that these conditions are not LMI but only a test of positive definiteness of a given matrix, which makes their method efficient from computational reasons. However, even though the necessary condition of stability is simple to test, the sufficient condition of stability is still numerically complicated to verify, because the estimated order assessing the sufficiency can be very high.

The objective of the present paper follows the same spirit of deriving necessary and sufficient stability conditions  of time-delay systems but with a different objective. Indeed, we aim at using the approach developed in \cite{Mondie2021,Gu2013} to prove that the sufficient LMI conditions arising from the Bessel-Legendre inequality are asymptotically necessary, as the degree $n$ of the Legendre polynomial to be considered increases. This corresponds to answer the following questions. 
\begin{itemize}
    \item If a time-delay system is stable, does there exist an order $N^\ast$ for which some LMI stability conditions are necessary satisfied?
    \item Is it possible to estimate this order $N^\ast$?
\end{itemize}

The paper is organized as follows. Section~\ref{sec1} formulates the problem of stability analysis of LTI systems subject to a single point-wise constant delay. After presenting a brief summary of the properties of Legendre polynomials, sufficient scalable LMI stability conditions are proposed. Then, Section~\ref{sec2} proves the converse theorem of this first result, by ensuring the existence of an order $N^\ast$ for which the LMI stability conditions must be verified for stable time-delay systems. Section~\ref{sec3} goes beyond the existence of an order by providing an estimation of the order $N^\ast$ of the polynomial approximation for which the LMI stability conditions are guaranteed. Finally, the theoretical results are evaluated on academic examples leading to several discussions. The paper ends with a conclusion and several appendices gathering the technical proofs required to derive the main results.

\vspace{0.2cm}

\emph{Notations :} Throughout the paper, $\mathbb{N}$ ($\mathbb{N}^{\ast}$), $\mathbb{R}^{m\times p}$ and $\mathbb{S}^m$ ($\mathbb{S}^m_+$) denote the sets of natural numbers (excluding zero), real matrices of size $m\times p$ and symmetric matrices of size $m$ (positive definite), respectively. For any $x$ in $\mathbb{R}$, $\lceil x\rceil$ stands for the least integer greater than or equal to $x$. For any square matrix $M\in\mathbb{R}^{m\times m}$, $M(p,q)$ denotes the entries of $M$ located at the $p^{th}$ row and $q^{th}$ column, $M^\top$ denotes the transpose of $M$, $\mathcal{H}(M)$ stands for $M+M^\top$ and inequality $M\succ 0$ means that $M\in\mathbb{S}^m_+$. Furthermore, for any $M$ in $\mathbb{S}^m$, its minimal and maximal eigenvalues are denoted $\sigmaD{M}$ and $\sigmaU{M}$. The 2-norm of matrix $M$ in $\mathbb{R}^{m\times p}$ is $\norm{M}=\sqrt{\sigmaU{M^\top M}}$. The vector $u=\mathrm{vec}(M)$ in $\mathbb{R}^{mp\times 1}$ collocates the columns of matrix $M$ in $\mathbb{R}^{m\times p}$ and the inverse of this operation is denoted $\mathrm{vec}^{-1}$ and is such that $\mathrm{vec}^{-1}(\mathrm{vec}(M))=M$.  Moreover, notation $\otimes$ represents the Kronecker product so that, for any matrices $(M_1,M_2)$ in $\mathbb{R}^{m_1\times p_1}\times \mathbb{R}^{m_2\times p_2}$, $M_1\otimes M_2=\begin{bsmallmatrix} M_1(1,1)M_2 & \dots &  M_1(1,p_1)M_2\\\vdots & \ddots & \vdots\\  M_1(m_1,1)M_2 & \dots & M_1(m_1,p_1)M_2 \end{bsmallmatrix}$ in $\mathbb{R}^{m_1m_2\times p_1p_2}$. The set of square-integrable functions from $(a,b)$ to $\mathbb{R}^{m\times p}$ is noted $\mathcal{L}_2(a,b;\mathbb{R}^{m\times p})$. Let finally $\mathcal{C}_{pw}(a,b;\mathbb{R}^{m\times p})$ be the set of piecewise continuous functions with a finite number of discontinuity points.

\section{Sufficient LMI stability conditions for time-delay systems}\label{sec1}

\subsection{System data}

Consider a LTI delay system given by
\begin{subequations}\label{eq:tds}
    \begin{empheq}{align}
        \dot{x}(t) &= A x(t) + A_d x(t-h),\; \forall t\geq 0,\label{eq:tdsa}\\
        x(t) &= \varphi(t)\in\mathcal{C}_{pw}(-h,0;\mathbb{R}^{n_x}), \; \forall t\in[-h,0],
    \end{empheq}
\end{subequations}
where the single delay $h>0$ and matrices $A,A_d$ in $\mathbb{R}^{n_x\times n_x}$ are constant and known. Without loss of generality, matrix $A_d$ is decomposed into a product $BC$ as 
\begin{equation}\label{eq:constC1}
A_d=BC,\mbox{ with } |C|=1,
\end{equation}
with $B,C^T\in\mathbb{R}^{n_x\times n_z}$ being full column rank matrices, with $n_z$ being the rank of matrix $A_d$. Along the paper, the Shimanov notation $x_t(\theta)=x(t+\theta)$, for all $(t,\theta)$ in $\mathbb R^+\times [-h,0]$ is adopted.

\begin{definition}[GES]
The trivial solution $x(t)\equiv 0$ of system~\eqref{eq:tds} is globally exponentially stable (GES), if there exist $\kappa\geq 1$ and $\mu>0$ such that the solution to  \eqref{eq:tds} generated by any initial condition $\varphi\in \mathcal{C}_{pw}(-h,0;\mathbb{R}^{n_x})$, denoted as $x(t,\varphi)$ verifies

\begin{equation}\label{def:GES}
\norm{x(t,\varphi)}\leq \kappa e^{-\mu t}\underset{\theta\in[-h,0]}{\mathrm{sup}}{\norm{\varphi(\theta)}},\quad \forall t\geq 0.
\end{equation}
\end{definition}

Several ways of assessing GES of LTI systems have been provided in the literature as mentioned in the introduction. Here, the contribution focuses on conditions arising from the application of the Lyapunov-Krasovskii theorem~\cite{Gu2003},  adapted to the LTI case.
\begin{theorem}[Lyapunov-Krasovskii theorem]\label{thm:LK}
If there exist positive scalars $\eps_1$, $\eps_2$ and $\eps_3$ and a continuous and differentiable functional $\mathcal{V}$ defined from $\mathcal{C}_{pw}(-h,0;\mathbb{R}^{n_x})$ to $\mathbb{R}$ such that, for any $\varphi$ in $\mathcal{C}_{pw}(-h,0;\mathbb{R}^{n_x})$, the following inequalities hold
\begin{itemize}
\item[(i)] $\eps_1 |\varphi(0)|^2\leq \mathcal{V}(\varphi)\leq \eps_2\underset{\theta\in[-h,0]}{\mathrm{sup}}{\norm{\varphi(\theta)}}^2$ and
\item[(ii)] $\dot{\mathcal{V}}(\varphi)\leq -\eps_3|\varphi(0)|^2$, where $\dot{\mathcal{V}}$ denotes here the derivative of $\mathcal{V}$ along the trajectories of~\eqref{eq:tds},
\end{itemize}
then, the trivial solution of system~\eqref{eq:tds} is GES.
\end{theorem}


The main underlying idea of this theorem is to determine a positive definite functional $\mathcal{V}$, such that its derivative with respect to time along the trajectories of the system (\ref{eq:tds}) is negative definite. The main problem within the application of this theorem is to design a suitable functional and then to provide some conditions that guarantee its positive definiteness and the negative definiteness of its derivative.
\newline
The derivation of stability conditions using the Lyapunov-Krasovskii theorem usually involves quite elaborate developments. To give an idea of the procedure involved in this approach and to provide a glimpse of its technical flavor, we present here some basics on the procedure to follow in order to derive asymptotic stability criteria for time-delay systems expressed in terms of LMI. The procedure follows three steps.
 
\underline{Step 1.} \textit{To propose a candidate Lyapunov-Krasovskii functional $\mathcal{V}$, based on the structure of the system.}\newline
\underline{Step 2.} \textit{To compute the derivative of the functional along the trajectories of the system.}\newline
\underline{Step 3.} \textit{To apply over-approximation technique to $\dot{\mathcal{V}}$ and derive a stability condition expressed in terms of LMI.}

Among the numerous methods employed in the literature, we will focus here on the method based on the application of the Bessel-Legendre integral inequality~\cite{Seuret2014stability,Seuret2015Hierarchy}, which is scalable with respect to the degree $n$ of the Legendre polynomial to be considered. Interestingly, this method leads to LMI admitting a hierarchical structure with respect to $n$. Before stating this theorem, let us first recall the main definitions and some characteristics about these polynomials.


\subsection{Definition of Legendre polynomials}

\begin{definition}
The Legendre polynomials are defined over the normalized interval $[0,1]$ as
\begin{equation} \label{eq:lk} 
    \forall k\in\mathbb{N},\quad
    l_k(\theta) = \overset{k}{\underset{j=0}{\sum}}(\begin{smallmatrix}k\\j\end{smallmatrix})(\begin{smallmatrix}k+j\\j\end{smallmatrix})(\theta-1)^j,
\end{equation}
where $(\begin{smallmatrix}k\\j\end{smallmatrix})$ stands for the binomial coefficients.
\end{definition}

The Legendre polynomials have been  widely used in the polynomial approximation theory~\cite{boyd2001} because they form an orthogonal sequence with respect to the inner product $\int_0^1 \phi^\top(\theta)\psi(\theta)\dth$, for any $\phi,\psi$ in $\mathcal{L}_2(0,1;\mathbb{R}^{m})$. Before going into the details, for any order $n$ and $m$ in $\mathbb{N}^{\ast}$ and $\theta\in[0,1]$, let us introduce the following notations:
\begin{align}\label{eq:ln}
&\ell_n^m(\theta)\!=\!\begin{bmatrix} l_0(\theta)&l_1(\theta)&\dots l_{n-1}(\theta)\end{bmatrix}^{\!\top}\!\!\!\otimes\! I_m\in\mathbb{R}^{nm\times m},\\
\label{eq:In}
&\mathcal{I}_n^m \!=\! \mathcal{I}_n^1\otimes I_m \in\mathbb{R}^{nm\times nm},\\
&\mathcal{I}_n^1({p,q}) \!=\! \left\{\begin{array}{lcl}
2p-1& \mbox{ if }p=q,\\
0 & \mbox{ otherwise,}
\end{array}
\right.\,
\forall(p,q)\in\{1,\dots, n\}^2.\nonumber
\end{align}
To ease the reading of the paper and where no confusion is possible, we will
omit the upper script $m=n_z$ and only use notations $\ell_n$ and $\mathcal{I}_n$.

The main motivation for employing these polynomials comes also from an efficient integral inequality, known as Bessel-Legendre inequality, which is stated here.
\begin{lemma}[Bessel-Legendre inequality]\label{lem:Bessel}
Let $ z\in\mathcal{L}_2(-h,0;\mathbb{R}^{n_z})$ and $S\in \mathbb{S}^{n_z}_+$ a positive definite matrix. The integral inequality
\begin{equation}\label{lem:ineqBessel}
    \int_{-h}^0 z^\top (\theta)S z(\theta)\dth \geq   \frac{1}{h}
    \zeta_n^\top( z) (\mathcal{I}_n^1\otimes S) \zeta_n( z),
\end{equation}
hold, for all $n\in \mathbb{N}^{\ast}$, where  
$ \zeta_n( z)= \int_{-h}^0 \ell_n\left(\frac{\theta+h}{h}\right) z(\theta)\mathrm d\theta$.
\newline
Moreover, if $ z$ is a polynomial function of degree $n-1$ over $(-h,0)$, the equality case holds
\begin{equation}\label{lem:eqBessel}
    \int_{-h}^0 z^\top (\theta)S z(\theta)\dth =   \frac{1}{h}
     \zeta_n( z)^T(\mathcal{I}_n^1\otimes S)  \zeta_n( z).
\end{equation}
\end{lemma}
\begin{proof} 
The proof is derived from \cite[Lemma~3]{Seuret2015Hierarchy} and is postponed to Appendix~\ref{app0}.
\end{proof}

\subsection{Sufficient LMI conditions}

We are now in position to state a stability theorem for system \eqref{eq:tds} based on the previous developments.
\begin{theorem}\label{thm:CSlmi}
If there exist an order $n$ in $\mathbb{N}^{\ast}$ and matrices $(\Pn,R,S)$ in $\mathbb{S}^{n_x+nn_z}\times\mathbb{S}^{n_z}_+\times\mathbb{S}^{n_z}_+$ such that the following LMI conditions hold
\begin{subequations}\label{eq:CSlmi}
    \begin{align}
    \Phi_n^+ = \Pn + \begin{bmatrix}0 & 0\\ 
    0 & \frac{1}{h} (\mathcal{I}_n^1\otimes S)\end{bmatrix}&\succ 0
    ,\label{eq:CSlmi1}\\
    \Phi_n^- = \begin{bmatrix}
       \mathcal{H}(\Pn\mathbf{A}_n) + \Psi(R,S) & \Pn \mathbf{B}_n \\
        \ast & -S
    \end{bmatrix}&\prec 0,\label{eq:CSlmi2}
    \end{align}
\end{subequations}
where 
\begin{equation*}
\begin{aligned}
    \Psi(R,S) &= \begin{bmatrix}C^\top(hR+S)C & 0\\ 0 & -\frac{1}{h}(\mathcal{I}_n^1\otimes R) \end{bmatrix},\\
    \mathbf{A}_n &= \begin{bmatrix} A & 0\\ \ell_n(1)C & -{\frac{1}{h}}\mathcal{L}_n^{n_z} \end{bmatrix},\;   \mathbf{B}_n = \ \begin{bmatrix} B \\ - \ell_n(0) \end{bmatrix},\\
    \ell_n(1) &= \begin{bsmallmatrix}I_{n_z}\\\vdots\\I_{n_z} \end{bsmallmatrix}, \; 
    \ell_n(0) = \begin{bsmallmatrix}I_{n_z}\\\vdots\\(-1)^{n-1}I_{n_z} \end{bsmallmatrix}\in\mathbb{R}^{nn_z},\\
    \mathcal{L}_n^{n_z} &= \mathcal{L}_n^1 \otimes I_{n_z}\in\mathbb{R}^{nn_z\times nn_z},\\
    \mathcal{L}_n^1(p,q) &= 
    \left\{
    \begin{array}{ll}
(2q-1)(1-(-1)^{p+q}) & \text{ if }\  p\geq q,\\
        0 & \text{otherwise},
    \end{array}
    \right.\\
    &\qquad\qquad\qquad\qquad\qquad\forall (p,q)\in\{1,\dots,n\}^2,
\end{aligned}
\end{equation*}
then, the trivial solution of system~\eqref{eq:tds} is GES.
\end{theorem}

\begin{proof}
For a given integer $n$ in $\mathbb{N}^\ast$, consider the functional defined as follows
\begin{align}\label{thm1:defV}
    \mathcal{V}_n(x_t)=&\begin{bmatrix} x_t(0)\\  \zeta_n( Cx_t)\end{bmatrix}^\top \Pn \begin{bmatrix} x_t(0)\\  \zeta_n( Cx_t)\end{bmatrix} \\
&+\displaystyle \int_{-h}^0x^\top_t(\theta) C^\top((\theta+h)R+S)Cx_t(\theta)\dth,\nonumber
\end{align}
where matrices $\Pn$, $R$ and $S$ are solution to the LMI stated in the theorem and where the augmented vector $ \zeta_n$ is defined as follows
\begin{equation}
 \zeta_n( Cx_t)= \int_{-h}^0 \ell_n(\theta) C x_t(\theta)\dth.
\end{equation}
Applying the Bessel-Legendre inequality~\eqref{lem:ineqBessel} yields 
\begin{equation*}
    \mathcal{V}_n(x_t)\geq \begin{bmatrix} x_t(0)\\  \zeta_n( Cx_t)\end{bmatrix}^\top \Phi_n ^+\begin{bmatrix} x_t(0)\\  \zeta_n( Cx_t)\end{bmatrix}.
\end{equation*}
Therefore, if condition \eqref{eq:CSlmi1} is verified, then there exists a sufficiently small $\eps_1>0$ such that $\mathcal V_n(x_t)\geq \eps_1 \norm{x_t(0)}^2$.
In addition, since $\mathcal{V}$ is quadratic with respect to $x_t$, selecting $\eps_2=\sigmaU{\Pn}+\frac{h^2}{2}\sigmaU{R}+h\sigmaU{S}$ ensures that inequality $\mathcal{V}_n(x_t)\leq\eps_2\underset{\theta\in[-h,0]}{\mathrm{sup}}{\norm{x_t(\theta)}}^2$ holds.
As in \cite{Seuret2014stability}, computing the derivative of the functional along the trajectories of the system yields
\begin{equation}\label{eq:dotV}
\begin{array}{lcl}
\dot{\mathcal{V}}_n(x_t)\!=\!2\!\begin{bmatrix} x_t(0)\\ \! \zeta_n( Cx_t)\!\end{bmatrix}^{\!\!\top}  \!\!\! \Pn \! \begin{bmatrix} \dot x_t(0)\\ \!\dot{\zeta}_n( Cx_t)\!\end{bmatrix}\!\! - \!x^\top_t\! (-h) C^\top\!\! S Cx_t(-h)\\
\quad \displaystyle+x^\top_t\!(0) C^\top\!(hR\!+\!S)Cx_t(0)
\!-\! \int_{-h}^0\!\!x^\top_t\!(\theta) C^\top\! RCx_t(\theta)\dth .
\end{array}
\end{equation}
Then, the dynamics~\eqref{eq:tdsa} and an integration by parts of $\dot{\zeta}_n( Cx_t)$ provide an expression $\begin{bsmallmatrix} \dot{x}_t(0)\\ \dot{\zeta}_n( Cx_t)\end{bsmallmatrix}$ with respect to $x_t(0)$, $ \zeta_n( Cx_t)$ and $x_t(-h)$. It ensures that
$$
\begin{bmatrix} 
\dot{x}_t(0)\\ 
\dot{\zeta}_n( Cx_t)
\end{bmatrix}=  \mathbf{A}_n \begin{bmatrix} 
x_t(0)\\ 
 \zeta_n( Cx_t)
\end{bmatrix}+ \mathbf{B}_n C x_t(-h).
$$ 
Re-injecting this expression into \eqref{eq:dotV} and applying again the Bessel-Legendre inequality~\eqref{lem:ineqBessel} yields
$$
\dot {\mathcal V}_n(x_t)\leq  
\begin{bsmallmatrix} 
x_t(0)\\ 
 \zeta_n( Cx_t)\\
C x_t(-h)
\end{bsmallmatrix}^\top \Phi_n^- \begin{bsmallmatrix} 
x_t(0)\\ 
 \zeta_n( Cx_t)\\
Cx_t(-h)
\end{bsmallmatrix}.
$$
Therefore, if condition~\eqref{eq:CSlmi2} is satisfied, then there exists a sufficiently small $\eps_3$ such that $\dot {\mathcal V}_n(x_t)\leq -\eps_3 |x_t(0)|^2$, which concludes the proof by application of the Lyapunov-Krasovskii theorem.
\end{proof}

Note that the proposed theorem has already been presented in \cite{Seuret2014stability} or in \cite{Seuret2015Hierarchy}. As in \cite{Seuret2014stability} or in \cite{Seuret2015Hierarchy}, the previous stability condition is hierarchical with respect to $n$. This aspect is presented formally in the following lemma.

\begin{lemma}\label{lem:hierarchy}
If there exists an integer $N$ in $\mathbb{N}^\ast$, for which a solution to LMI~\eqref{eq:CSlmi} exists, then there also exists a solution to the same problem for any integer $n\geq N$.
\end{lemma}
\begin{proof}
The proof uses similar arguments to the ones provided in~\cite[Theorem~7]{Seuret2015Hierarchy}. A glimpse of the proof consists in introducing $\mathbf{P}_{n+1}=\begin{bsmallmatrix}\Pn&0\\0&0\end{bsmallmatrix}$, so that $\mathcal{V}_{n+1}(x_t)=\mathcal V_{n}(x_t)$ and so that one can exhibit a solution to the LMI problem at order $n+1$ based on the solution at order $n$. The details of the proof are omitted but strongly relies on the structure of $\Phi_n^+$ and $\Phi_n^-$.
\end{proof}

To sum up the results presented so far, Theorem~\ref{thm:CSlmi} presents scalable LMI conditions for the stability (GES) of LTI delay systems. These conditions is parameterized by $n$ in $\mathbb{N}^{\ast}$, corresponding to the degree of the Legendre polynomials considered in the construction of the candidate Lyapunov-Krasovskii functional. Moreover, it is demonstrated that when $n$ increases, the conservatism can only be reduced. It is however legitimate to wonder if a converse result can be proven. This direction refers to the possibility of assessing the satisfaction of these LMI for a sufficiently large $n$, when the system under consideration is known a priori to be GES for a given delay $h$.  Apart the convergence, i.e. the existence of an order $N^\ast$ for which these LMI conditions are guaranteed, an interesting underlying question is to estimate analytically such an order. The next developments aim at providing a solution to both problems. More specifically, the next section presents a converse theorem with a proof of existence of $N^\ast$, while the latter address the problem of estimating order $N^\ast$.

\section{Necessity of LMI stability conditions: Existence of an order}\label{sec2}

In this section, the objective is to prove the converse side of Theorem~\ref{thm:CSlmi}. In other words, we aim at proving that, if the trivial solution of system~\eqref{eq:tds} is~GES, then there exists an order $N^\ast$ in $\mathbb{N}^\ast$ from which LMI conditions in~\eqref{eq:CSlmi} hold.
The developments go against the grain of usual approaches and are inspired by the converse Lyapunov-Krasovskii theorem~\cite{Wenzhang1989,Zhabko2015}, stated below.

\begin{theorem}[Converse Lyapunov-Krasovskii theorem]\label{thm:converse}
If the trivial solution of system \eqref{eq:tds} is GES, then there exist positive scalars $\eps_1$, $\eps_2$ and $\eps_3$ and a continuous and differentiable functional $\mathcal{V}$ defined from $\mathcal{C}_{pw}(-h,0;\mathbb{R}^{n_x})$ to $\mathbb{R}$ such that, for any $\varphi$ in $\mathcal{C}_{pw}(-h,0;\mathbb{R}^{n_x})$, inequalities
\begin{itemize}
\item[(i)] $\eps_1 |\varphi(0)|^2\leq \mathcal{V}(\varphi)\leq \eps_2\underset{\theta\in[-h,0]}{\mathrm{sup}}{\norm{\varphi(\theta)}}^2$ hold and
\item[(ii)] $\dot{ \mathcal{V}}(\varphi)\leq -\eps_3|\varphi(0)|^2$ holds, where $\dot{\mathcal{V}}$ denotes here the derivative of $\mathcal{V}$ along the trajectories of~\eqref{eq:tds},
\end{itemize}
\end{theorem}

First, the so-called complete Lyapunov-Krasovskii functional is introduced. It fulfills the converse Lyapunov-Krasovskii theorem statement and serves as a target functional. Then, a link is established between such a complete functional and the functional used in the previous section with a particular structure. The two LMI conditions in~\eqref{eq:CSlmi} are consequently obtained. On one side, assuming that system~\eqref{eq:tds} is~GES, one ensures that~\eqref{eq:CSlmi1} holds, for any order $n$. On the other side, the time-derivative of the functional along the trajectories of system~\eqref{eq:tds} becomes negative as the order increases so that~\eqref{eq:CSlmi2} holds, for sufficiently large orders. As a final step, the new converse theorem is stated and proven.

\underline{Step 1.} \textit{To introduce the complete Lyapunov-Krasovskii functional which can be expressed analytically with respect to a Lyapunov matrix function.}\newline
\underline{Step 2.} \textit{To make the link between the functional $\mathcal{V}_n$ introduced in~\eqref{thm1:defV} and the Legendre approximation of the complete functional.}\newline
\underline{Step 3.} \textit{To assess that the condition~\eqref{eq:CSlmi1} holds, assuming that system~\eqref{eq:tds} is GES.}\newline
\underline{Step 4.} \textit{To assess that the condition~\eqref{eq:CSlmi2} holds, for a sufficiently large order.}\newline
\underline{Step 5.} \textit{To conclude on the necessity of LMI conditions given in Theorem~\ref{thm:CSlmi}.}
\newline
The section is decomposed in five subsections respectively based on the steps listed above.

\subsection{The complete Lyapunov-Krasovskii functional}

In this subsection, based on~\cite{Kharitonov2013}, we introduce the so-called complete Lyapunov-Krasovskii functional, which will be used in conjonction with Theorem~\ref{thm:converse}.

\begin{definition}\label{defn:U}
For any  $(W_1,W_2,W_3)$ in $\mathbb{S}^{n_x}_+\times \mathbb{S}^{n_z}_+\times \mathbb{S}^{n_z}_+$, the complete Lyapunov-Krasovskii functional $\mathcal{V}$ is 
\vspace{-0.2cm}
\begin{align}\label{eq:V0}
    &\mathcal{V}(x_t) \!=\! x_t^\top(0)U(0)x_t(0) \!+\!\! \displaystyle2x_t^\top(0) \!\!\int_{-h}^0 \!\!\!\!U(\theta\!+\!h)BCx_t(\theta)\dth\nonumber \\
    \!\!&\quad\displaystyle+ \int_{-h}^0\! \int_{-h}^0\!\! x_t^\top(\theta_1)(BC)^\top\! U(\theta_2-\theta_1)BCx_t(\theta_2)\dth_1\dth_2\nonumber \\
    &\quad \displaystyle+ \int_{-h}^0x^\top _t(\theta) C^\top((\theta+h)W_2+W_3)Cx_t(\theta)\dth ,
\end{align}
where $U$ is the Lyapunov matrix defined from $[-h,h]$ to $\mathbb{R}^{n_x\times n_x}$ as the solution of the following matrix delayed differential equation
\begin{equation}
\left\{
\begin{aligned}
    &U'(\theta) = - U(\theta) A - U(\theta+h)A_d,\quad \forall \theta \leq 0,\\
    & A^\top U(0) \!+\! A_d^\top U(h) \!+\! U(0) A \!+\! U(-h) A_d = -W,\\
    &\;U(\theta) = U^\top(-\theta),\qquad\qquad \forall \theta>0,
\end{aligned}
\right.
\end{equation}
where 
\begin{equation}\label{def:W}
   W = W_1+C^\top \big( hW_2 + W_3 \big) C.
\end{equation}
\end{definition}
\vspace{-0.5cm}
Interestingly, using $\mathrm{vec}(M_1M_2M_3)\!=\!(M_3^\top\!\otimes\! M_2)\mathrm{vec}(M_2)$, for any matrices $M_1,M_2,M_3$ in $\mathbb{R}^{n_x\times n_x}$, a technique developed in~\cite[Section 2.10]{Kharitonov2013} provides an analytic expression of $U$ given by
\vspace{-0.1cm}
\begin{equation}\label{eq:U0}
U(\theta) =
    \left\{\begin{array}{lcl}\!
     \mathrm{vec}^{-1}\left(\begin{bsmallmatrix}I_{n_x^2}&0\end{bsmallmatrix}e^{\theta\mathcal{M}}\mathcal{N}^{-1}\begin{bsmallmatrix}0\\-\mathrm{vec}(W)\end{bsmallmatrix}\right) & \text{if} & \theta \leq 0,\\
     U^\top(-\theta) & \text{if} & \theta > 0,
    \end{array}\right.
\end{equation}
\vspace{-0.1cm}
where $W$ is given in \eqref{def:W} and
\begin{align}\label{eq:MN0}
    \mathcal{M} &= 
    \begin{bsmallmatrix}
        -A^\top \otimes I_{n_x} & -A_d^\top \otimes I_{n_x}\\
        I_{n_x} \otimes A_d^\top & I_{n_x} \otimes A^\top
    \end{bsmallmatrix},\\
    \mathcal{N} &= 
    \begin{bsmallmatrix}
        I_{n_x^2} & 0\\
        \!A^\top\otimes I_{n_x} + I_{n_x}\otimes A^\top & A_d^\top \otimes I_{n_x}\!
    \end{bsmallmatrix} \!+\!
    \begin{bsmallmatrix}
        0 & -I_{n_x^2}\!\\
        \!I_{n_x} \otimes A_d^\top& 0
    \end{bsmallmatrix} e^{-h\mathcal{M}}.\nonumber
\end{align}
It is well-known that the complete functional $\mathcal{V}$ satisfies the following lemma.
\begin{lemma}\label{lem:dV}
For any $(W_1,W_2,W_3)$ in $\mathbb{S}^{n_x}_+\times \mathbb{S}^{n_z}_+\times \mathbb{S}^{n_z}_+$, the time-derivative along the trajectories of system~\eqref{eq:tds} of the complete Lyapunov-Krasovskii functional $\mathcal{V}$ given by~\eqref{eq:V0} yields
\begin{equation}\label{eq:dV}
\dot{\mathcal{V}}(x_t) \!=\! -\frac{1}{h}\!\! \int_{-h}^0\! \begin{bsmallmatrix}x_t(0)\\hCx_t(\theta)\\Cx_t(-h) \end{bsmallmatrix}^{\!\top}\!\! \begin{bsmallmatrix}W_1&0&0\\\ast&\frac{1}{h}W_2&0\\\ast&\ast&W_3\end{bsmallmatrix}\! \begin{bsmallmatrix}x_t(0)\\hCx_t(\theta)\\Cx_t(-h)\end{bsmallmatrix}.
\end{equation}
\end{lemma}
\begin{proof}
The proof is provided in~\cite[Theorem 2.11]{Kharitonov2013}.
\end{proof}

\vspace{-0.4cm}
\begin{remark}\label{rem:lyapcond}
For any symmetric positive definite matrices $(W_1,W_2,W_3)$ and by assuming that there is no characteristic root of~\eqref{eq:tds} such that its opposite is also characteristic root, it is worth mentioning that this functional exists and is the unique solution of~\eqref{eq:dV}. For more details, one can refer to~\cite[Theorems 2.8 and 2.10]{Kharitonov2013}.
\end{remark}

It is also well-known that the complete funtional $\mathcal{V}$ is positive definite if the trivial solution of system~\eqref{eq:tds} is GES as demonstrated in~\cite[Theorem 5.19]{Gu2003}. In light of paper purpose, a variation of this result is proposed, which will corresponds exactly to the present framework. 
\begin{lemma}\label{lem:CN2n}
Consider $(W_1,W_2,W_3)$ in $\mathbb{S}^{n_x}_+\times \mathbb{S}^{n_z}_+\times \mathbb{S}^{n_z}_+$.
If the trivial solution of system~\eqref{eq:tds} is GES, there exists a positive scalar $\varepsilon>0$ such that the complete functional $\mathcal{V}$ given by \eqref{eq:V0} satisfies
\begin{equation}\label{eq:Vnineq1}
\begin{aligned}
     &\mathcal{V}(\varphi) - \int_{-h}^0(\theta+h)\varphi^\top(\theta) C^\top W_2C\varphi(\theta)\dth\\
     &\qquad\qquad\qquad \geq  \displaystyle\!\varepsilon \left(\norm{\varphi(0)}^2 \!+\!\int_{-h}^0|C\varphi(\theta)|^2\dth\right)\!,
\end{aligned}
\end{equation}
for any $\varphi$ in $\mathcal{C}_{pw}(-h,0;\mathbb{R}^{n_x})$.
\end{lemma}
\begin{proof}
The proof is postponed to Appendix~\ref{app3}.
\end{proof}

Lastly, the Lyapunov matrix~$U$ introduced in Definition~\ref{defn:U} satisfies several properties. The first one is related to its continuity. The second property concerns its derivative $U'$, which is only continuous over $(0,h]$ and $[-h,0)$ and has a discontinuity at $0$, which is measured as follows.

\begin{lemma}\label{prop:DeltaU}
The Lyapunov matrix~$U$ given by~\eqref{eq:U0} verifies the following properties.
\vspace{-0.1cm}
\begin{itemize}
\item[(i)] The Lyapunov matrix $U$ is continuous on $[-h,h]$ and
\vspace{-0.1cm}
\begin{equation*}
    U(0)=U^\top(0).
\end{equation*}
\item[(ii)]
The Lyapunov matrix $U$ is infinitely differentiable on $(0,h]$ and $[-h,0)$. Its derivative has a discontinuity only at $0$ so that
\vspace{-0.1cm}
\begin{equation*}
    \Delta U'(0):= \lim_{\epsilon\rightarrow 0}  \left( U'(\epsilon) - U'(-\epsilon) \right) =  W.
\end{equation*}
\end{itemize}
\end{lemma}
\begin{proof}
The two items are proved in~\cite[Lemma 2.4 and Lemma 2.6]{Kharitonov2013}, respectively.
\end{proof}

These regularity conditions satisfied by the Lyapunov matrix have been used for $\mathcal{H}_2$~\cite{Jarlebring2011} or~$\mathcal{H}_\infty$~\cite{Kharitonov2002} analysis. They will also be at the heart of the derivation of the convergence results. 

To summarize, this subsection emphasized the structure of the complete functional whose kernels are defined by the Lyapunov matrix $U$. This particular structure has the benefits of ensuring by construction that its time-derivative along the trajectories of the system verifies~\eqref{eq:dV} and that the positive definiteness is ensured for GES systems.

\subsection{Construction of matrices $\mathbf{P}_n$, $R$ and $S$}

In this subsection, the objective is to understand how to relate the functional defined in \eqref{thm1:defV} and the complete one in \eqref{eq:V0}, for a particular structure of matrices $(\mathbf{P}_n,R,S)$.
To do so, the main idea is to exploit the terms of~\eqref{eq:V0} that are expressed in $U(\theta+h)B$ and $B^\top U(\theta_2-\theta_1)B$. The Legendre polynomial approximation of these functions at any order $n$ in $\mathbb{N}^{\ast}$ writes
\begin{equation}\label{eq:Untilde0}
\begin{array}{rcll}
U(\theta+h) B&=&  \mathbf{U}_{1,n}  \ell_n\left(\frac{\theta+h}{h}\right)+\tilde{U}_{1,n}(\theta), \, &\forall \theta\in[-h,0],\\
B^\top U(\theta) B&=&   \mathbf{U}_{2,n} \ell_n\!\left(\frac{\theta+h}{2h}\right)  +\tilde{U}_{2,n}(\theta) ,\, &\forall \theta\in[-h,h].
\end{array}
\end{equation}
In this decomposition, the constant matrices $\mathbf{U}_{1,n}$ and $\mathbf{U}_{2,n}$ have been selected as the orthogonal projection of $U(\theta+h) B$ and $B^\top U(\theta) B$, respectively, on the $n$ first Legendre polynomials $\ell_n$. Their expressions are given by
\begin{equation}\label{eq:Unapprox}
\begin{aligned}
    \mathbf{U}_{1,n} &=  \frac{1}{h}\left( \int_{-h}^0 U(\theta+h)B\ell_n^\top\left(\frac{\theta+h}{h}\right)\dth \right)\mathcal{I}_n,\\
    \mathbf{U}_{2,n} &=  \frac{1}{2h}\left(\int_{-h}^h B^\top U(\theta)B\ell_n^\top\left(\frac{\theta+h}{2h}\right)\dth\right)\mathcal{I}_n.
\end{aligned}
\end{equation}
Functions $\tilde{U}_{1,n}(\theta)$ and $\tilde{U}_{2,n}(\theta)$ can be interpreted as the approximation errors of the orthogonal projections, and verify
$$
\begin{array}{r}
\displaystyle \int_{-h}^0\!\!\! \tilde{U}_{1,n}(\theta)\ell_n^\top\!\!\left(\!\frac{\theta\!+\!h}{h}\!\right)\!\dth\!=\!
\underbrace{ \int_{-h}^0\!\! U(\theta\!+\!h) B\ell_n^\top\!\left(\!\frac{\theta\!+\!h}{h}\!\right)\!\dth}_{=\mathbf{U}_{1,n}(\mathcal{I}_n/h)^{-1}}\\
\qquad \qquad \displaystyle- \mathbf{U}_{1,n} \underbrace{\int_{-h}^0     \ell_n\left(\!\frac{\theta\!+\!h}{h}\!\right)\ell_n^\top\left(\!\frac{\theta\!+\!h}{h}\!\right)\dth}_{=(\mathcal{I}_n/h)^{-1}}=0.
\end{array}
$$ 
Similarly, the same calculations ensure that error $\tilde{U}_{2,n}$ is orthogonal to the $n$ first Legendre polynomials considered over $[-h,h]$, i.e. $$\int_{-h}^h \tilde{U}_{2,n}(\theta)\ell_n^\top\left(\frac{\theta+h}{2h}\right)\dth=0.$$

The next developments aim at demonstrating the uniform convergence of the polynomial approximation. Following the theory of polynomial approximation (see for instance~\cite{boyd2001} and references therein), it results that $\tilde{U}_{1,n}$ and $\tilde{U}_{2,n}$ given in~\eqref{eq:Untilde0} converge to zero in the sense of the $\mathcal{L}_2$ norm. This is actually a by-product of Bessel-Legendre inequality. This implies that the approximation errors converges to zero almost everywhere on  their domain of definition, as $n$ tends to infinity. Nevertheless, uniform convergence properties can even be obtained using the regularity of the Lyapunov matrix.

\begin{lemma}\label{lem:convergence}
Consider the Lyapunov matrix $U$ for the time-delay system \eqref{eq:tds} defined for any $(W_1,W_2,W_3)$ in $\mathbb{S}^{n_x}_+\times \mathbb{S}^{n_z}_+\times \mathbb{S}^{n_z}_+$. The following statements hold.
\vspace{-0.3cm}
\begin{itemize}
\item [(i)] The approximated Legendre function $\mathbf{U}_{1,n}\ell_n\left(\frac{\theta+h}{h}\right)$ converges uniformly to $U(\theta+h)B$, when $n$ tends to infinity, on the closed interval $[-h,0]$. 
\newline More precisely, for any $n\geq4$, the following inequality holds 
\begin{equation}\label{eq:Wang}
 \underset{\theta\in[-h,0]}{\mathrm{sup}}\!\norm{\tilde{U}_{1,n}(\theta)} \leq \bar{u}_{1,n}:=\frac{\varrho_1}{\sqrt{n-3}} \norm{W},
\end{equation}
where 
\vspace{-0.2cm}
\begin{equation}\label{eq:r1}
    \varrho_1  = \sqrt{n_x} \left(\frac{\pi}{2}\right)^\frac{3}{2}e^{h\norm{\mathcal{M}}}\norm{\mathcal{M}^2\mathcal{N}^{-1}}h^2 \norm{B},\;
\end{equation}
and matrices $\mathcal{M}$, $\mathcal{N}$ and $W$ are given by~\eqref{eq:MN0}.
\item[(ii)] The approximated Legendre function $\frac{1}{h}\mathbf{U}_{1,n} \ell_n^\prime \left(\frac{\theta+h}{h}\right)$ converges uniformly to $U^\prime(\theta+h)B$, when $n$ tends to infinity, on the interval $[-h,0)$. 
\newline  More precisely, for any $n\geq6$, the following inequality holds 
\begin{equation}\label{eq:Wang2}
  \underset{\theta\in[-h,0)}{\mathrm{sup}} \norm{\tilde{U}^\prime_{1,n}(\theta)} \leq \bar{u}_{2,n}:=\frac{\varrho_2}{\sqrt{n-5}}\norm{W},
\end{equation}
where 
\vspace{-0.2cm}
\begin{equation}\label{eq:r2}
    \varrho_2  = \frac{1}{2}\sqrt{n_x} \left(\frac{\pi}{2}\right)^\frac{3}{2}e^{h\norm{\mathcal{M}}}\norm{\mathcal{M}^4\mathcal{N}^{-1}}h^3\norm{B},\;
\end{equation}
and matrices $\mathcal{M}$, $\mathcal{N}$  and $W$ are given by~\eqref{eq:MN0}.
\item[(iii)] The approximated Legendre function $\mathbf{U}_{2,n} \ell_n(\frac{\theta+h}{2h})$ converges uniformly to $B^\top U(\theta)B$, when $n$ tends to infinity, on the closed interval $[-h,h]$. 
\newline  More precisely, for any $n\geq4$, the following inequality holds 
\begin{equation}\label{eq:Wangbar}
  \underset{\theta\in[-h,h]}{\mathrm{sup}} \!\norm{\tilde{U}_{2,n}(\theta)}\! \leq \!      \bar{u}_{3,n}\!:=\!\frac{\varrho_3}{\sqrt{n-3}}\norm{W},
\end{equation}
where 
\vspace{-0.2cm}
\begin{equation}\label{eq:r3}
    \varrho_3  = \sqrt{2\pi} \big( 1 +  \sqrt{n_x}\pi h e^{h\norm{\mathcal{M}}}\norm{\mathcal{M}^2\mathcal{N}^{-1}}\big) h \norm{B}^2,\;
\end{equation}
and matrices $\mathcal{M}$, $\mathcal{N}$  and $W$ are given by~\eqref{eq:MN0}.
\end{itemize}
\end{lemma}

\begin{proof}
The proof follows the arguments provided in~\cite[Theorem~2.5]{Wang2012} and is based on the regularity properties of the Lyapunov matrix $U$ highlighted in Lemma~\ref{prop:DeltaU}. The proof is postponed to Appendix~\ref{app1}, for the sake of readability.
\end{proof}
The previous lemma provides uniform upper bounds on the error done by three approximations related to the Lyapunov matrix $U$. 
These uniform convergence results are only presented as the polynomial approximations of $U$ but addresses more general problems of functional analysis and approximation of continuous functions.

All these upper bounds depend explicitly on the arbitrary symmetric positive definite matrices $(W_1,W_2,W_3)$ through the term $\norm{W}$.

As a final result, Lemma \ref{lem:convergence} implies the following corollary.
\begin{corollary}\label{cor:bounds}
For any positive scalar $\eta>0$, there exists an integer $N^\ast_\eta$ such that 
\begin{equation}
\max\left( \bar{u}_{1,n}, \bar{u}_{2,n},\bar{u}_{3,n}\right) \leq \eta,\quad \forall n\geq N^\ast_\eta.
\end{equation}
\end{corollary}

\begin{proof} 
Since the upper bounds of approximation errors tend to $0$ as $n$ tends to infinity, the existence of such $N_\eta^\ast$ for any $\eta>0$ is guaranteed.
\end{proof}

We are now in position to construct matrices $(\mathbf{P}_n,R,S)$ for given matrices $(W_1,W_2,W_3)$. The guiding principle is to mimic the complete Lyapunov-Krasovskii functional, replacing $U(\theta+h)B$ and $B^\top U(\theta_2-\theta_1)B$ by their Legendre polynomial approximation given by~\eqref{eq:Untilde0} in order to benefit from its properties as formulated in the following lemma.

\begin{lemma}\label{lem:Vn}
For any $(W_1,W_2,W_3)$ in $\mathbb{S}^{n_x}_+\times \mathbb{S}^{n_z}_+\times \mathbb{S}^{n_z}_+$, define the functional $\mathcal{V}_n$ by~\eqref{thm1:defV} with matrices
\begin{equation}\label{lem:VnDef:Pn}
\begin{array}{l}
    \Pn=\begin{bmatrix}
    U(0) & \mathbf{U}_{1,n}\\
    \ast & \mathbf{T}_{n}
    \end{bmatrix}, \quad  R=W_2,\quad S=W_3,\\
    \mathbf{T}_n \!=\!\! \displaystyle \!\!\!\overset{\ 0\ \ 0}{\underset{-h\ -h}{\iint}}\!\!\! \frac{\mathcal{I}_n}{h}\ell_n\!\!\left(\!\frac{\theta_1\!+\!h}{h}\!\right)\!\mathcal{U}_{2,n}(\theta_2\!-\!\theta_1\!)\ell_n^\top \!\!\left(\!\frac{\theta_2\!+\!h}{h}\!\right)\!\frac{\mathcal{I}_n}{h}\dth_{\!1}\dth_{\!2},\\
    \mathcal{U}_{2,n}(\theta) \!=\!\displaystyle \mathbf{U}_{2,n} \ell_n\!\left(\frac{\theta+h}{2h}\right),
\end{array}
\end{equation} 
where $\mathbf{U}_{1,n}$, $\mathbf{U}_{2,n}$ are given in \eqref{eq:Unapprox}. 
The time-derivative of this functional along the trajectories of system~\eqref{eq:tds} yields
\begin{equation}
    \dot{\mathcal{V}}_n(x_t) = -\frac{1}{h} \int_{-h}^0 \begin{bsmallmatrix}x_t(0)\\hCx_t(\theta)\\Cx_t(-h) \end{bsmallmatrix}^{\!\top}\!\Psi_n(\theta) \begin{bsmallmatrix}x_t(0)\\hCx_t(\theta)\\Cx_t(-h)\end{bsmallmatrix},
\end{equation}
where matrix
\begin{equation}\label{eq:psin}
    \Psi_n(\theta)\!:=\!
    \begin{bsmallmatrix}
     W_1\!+\!\mathcal{H}(\tilde {U}_{1,n}(0)C) &
    \Psi_n^{1}(\theta) & -\tilde{U}_{1,n}(-h)\\
    \ast&\frac{1}{h}W_2&\Psi_n^{2}(\theta) \\
    \ast&\ast&W_3
    \end{bsmallmatrix} \!\succ\! 0,
\end{equation}
is defined for any $n$ in $\mathbb{N}^{\ast}$ and for all $\theta$ in $[-h,0]$ with
\begin{equation*}
\begin{array}{rcl}
\Psi_n^{1}(\theta) &=& A^\top\tilde{U}_{1,n}(\theta)- \tilde{U}_{1,n}^\prime(\theta)+C^\top\tilde{U}_{2,n}(\theta),\\
\Psi_n^{2}(\theta)&=& B^\top\tilde{U}_{1,n}^\top(\theta)- \tilde{U} _{2,n}^\top(\theta+h),
\end{array}
\end{equation*}
and with $\tilde{U}_{1,n}$, $\tilde{U}^\prime_{1,n}$ and $\tilde{U}_{2,n}$ being the approximation errors of the Lyapunov matrix $U$ of system \eqref{eq:tds}.
\end{lemma}
\begin{proof}
To begin with, re-injecting the expression of $U(\theta+h) B$ and $B^\top U(\theta) B$ and using their approximation given by~\eqref{eq:Untilde0} into the complete functional $\mathcal V$ leads to
\begin{equation}\label{eq:V1}
\begin{array}{l}
    \mathcal{V}(x_t) = x_t^\top\!(0)U(0)x_t(0)\\
    + \displaystyle 2x_t^\top\!(0)\mathbf{U}_{1,n} \!\!\int_{-h}^0\!\!\!\!  \ell_n\!\left(\!\frac{\theta\!+\!h}{h}\!\right)\! Cx_t(\theta)\dth \\
    \!\!\displaystyle + \int_{-h}^0\! \int_{-h}^0\!\!\! x_t^\top  (\theta_1)C^\top  \mathbf{U}_{2,n} \ell_n\!\left(\!\frac{\theta_2\!-\!\theta_1\!+\!h}{2h}\!\right)  Cx_t(\theta_2)\dth_1\dth_2\\
 \displaystyle+ \int_{-h}^0x^\top _t(\theta) C^\top(W_2+(\theta+h)W_3)Cx_t(\theta)\dth \\
    \displaystyle + 2 x_t^\top \int_{-h}^0  \tilde{U}_{1,n}(\theta)Cx_t(\theta)\dth \\
    \displaystyle+ \int_{-h}^0\! \int_{-h}^0 x_t^\top  (\theta_1)C^\top  \tilde{U}_{2,n}(\theta_2-\theta_1)Cx_t(\theta_2)\dth_1\dth_2.
\end{array}
\end{equation}
Since $\mathcal{U}_{2,n}(\theta_2-\theta_1)=\mathbf{U}_{2,n} \ell_n\!\left(\frac{\theta_2-\theta_1+h}{2h}\right)$ belongs to $\mathbb{R}^{n_z\times n_z}$ and is a polynomial function of degree $n-1$ in both $\theta_1$ and $\theta_2$, it  can be decomposed using the basis of Legendre polynomials  $\ell_n\left(\frac{\theta_1+h}{h}\right)$ and $\ell_n\left(\frac{\theta_2+h}{h}\right)$. Thanks to the orthogonality of Legendre polynomials, this decomposition writes
$$
\mathbf{U}_{2,n} \ell_n\!\left(\!\!\frac{\theta_2\!-\!\theta_1\!+\!h}{2h}\!\!\right)\!=\!\ell_n^\top\!\left(\!\frac{\theta_1\!+\!h}{h}\!\right) \mathbf{T}_n\ell_n\!\left(\!\frac{\theta_2\!+\!h}{h}\!\right),
$$
where matrix $\mathbf{T}_n$ is the symmetric matrix given in \eqref{lem:VnDef:Pn}. Note that the symmetry of $\mathbf{T}_n$ is ensured by the symmetry of $\tilde{U}_{2,n}$ highlighted in Property~\ref{Prop:U2nSym} that has been postponed in Appendix~\ref{app2} in order to ease the reading. Hence, using the same augmented vector as in  the proof of Theorem \ref{thm:CSlmi}, i.e. $ \zeta_n( Cx_t)= \int_{-h}^0  \ell_n\left(\frac{\theta+h}{h}\right)Cx_t(\theta)\dth$, functional $\mathcal{V}$ reduces to the following expression
\begin{equation}\label{eq:V2}
\begin{array}{lcl}
    \mathcal{V}(x_t) = \begin{bmatrix}
    x_t(0)\\
     \zeta_n( Cx_t)
    \end{bmatrix}^\top\begin{bmatrix}
    U(0) & \mathbf{U}_{1,n}\\
    \ast & \mathbf{T}_n
    \end{bmatrix}\begin{bmatrix}
    x_t(0)\\
     \zeta_n( Cx_t)
    \end{bmatrix}\\
    + \displaystyle \int_{-h}^0x^\top _t(\theta) C^\top(W_2+(\theta+h)W_3)Cx_t(\theta)\dth \\
    + \displaystyle 2 x_t^\top(0) \int_{-h}^0  \tilde{U}_{1,n}(\theta)Cx_t(\theta)\dth \\
    	\!\!\displaystyle+ \int_{-h}^0\! \int_{-h}^0 x_t^\top  (\theta_1)C^\top  \tilde{U}_{2,n}(\theta_2-\theta_1)Cx_t(\theta_2)\dth_1\dth_2.
\end{array}
\end{equation}
Therefore, by selecting $\Pn$, $R$ and $S$ as in~\eqref{lem:VnDef:Pn},  functional $\mathcal{V}_n$ in~\eqref{thm1:defV} is retrieved and we obtain
\begin{equation*}
\begin{array}{lcl}
    \mathcal{V}_n(x_t) = \mathcal{V}(x_t)
 \!-\! \displaystyle 2x_t^\top(0) \int_{-h}^0  \tilde{U}_{1,n}(\theta)Cx_t(\theta)\dth \\
    	\quad\displaystyle- \int_{-h}^0\! \int_{-h}^0 x_t^\top  (\theta_1)C^\top  \tilde{U}_{2,n}(\theta_2-\theta_1)Cx_t(\theta_2)\dth_1\dth_2.\\
\end{array}
\end{equation*}
Differentiating the previous expression leads to
\begin{equation*}\label{eq:dV1}
\begin{array}{lcl}
   \dot{\mathcal{V}}_n(x_t) = \dot{\mathcal{V}}(x_t)
    - \displaystyle 2\dot x_t^\top(0) \int_{-h}^0  \tilde{U}_{1,n}(\theta)Cx_t(\theta)\dth \\
    \quad - \displaystyle 2 x_t^\top(0) \int_{-h}^0  \tilde{U}_{1,n}(\theta)C\dot x_t(\theta)\dth \\
    \quad\displaystyle- \int_{-h}^0\! \int_{-h}^0 \dot x_t^\top  (\theta_1)C^\top  \tilde{U}_{2,n}(\theta_2-\theta_1)Cx_t(\theta_2)\dth_1\dth_2.\\
    \quad\displaystyle- \int_{-h}^0\! \int_{-h}^0 x_t^\top  (\theta_1)C^\top  \tilde{U}_{2,n}(\theta_2-\theta_1)C\dot x_t(\theta_2)\dth_1\dth_2.\\
\end{array}
\end{equation*}
Since $\mathcal{V}$ is the complete functional, its derivative is expressed using $W_1,W_2,W_3$. Several integrations by parts lead to the following expression of $\dot{\mathcal{V}}_n$
\begin{equation*}\label{eq:dV2}
\begin{array}{lcl}
   \dot{\mathcal{V}}_n(x_t) =\displaystyle -\frac{1}{h}\!\! \int_{-h}^0\! \begin{bsmallmatrix}x_t(0)\\hCx_t(\theta)\\Cx_t(-h) \end{bsmallmatrix}^{\!\top}\!\! \begin{bsmallmatrix}W_1&0&0\\\ast&\frac{1}{h}W_2&0\\\ast&\ast&W_3\end{bsmallmatrix}\! \begin{bsmallmatrix}x_t(0)\\hCx_t(\theta)\\Cx_t(-h)\end{bsmallmatrix}\\
	\quad - \displaystyle 2(A x_t(0)+BC x_t(-h))^\top \int_{-h}^0  \tilde{U}_{1,n}(\theta)Cx_t(\theta)\dth \\
    \quad - \displaystyle 2 x_t^\top(0)\left( \tilde{U}_{1,n}(0)C x_t(0)- \tilde{U}_{1,n}(-h)C x_t(-h)\right)\\
	\quad + \displaystyle  2x_t^\top(0) \int_{-h}^0  \tilde{U}^\prime_{1,n}(\theta)C x_t(\theta)\dth \\
	\quad - \displaystyle  x_t^\top(0) \int_{-h}^0  C^\top\underbrace{(\tilde{U} _{2,n}(\theta)+\tilde{U} _{2,n}^\top (-\theta))}_{=2\tilde{U} _{2,n}(\theta)}C x_t(\theta)\dth \\
	\quad + \displaystyle  x_t^\top(-h)\!\! \int_{-h}^0 \!\! C^\top \underbrace{(\tilde{U} _{2,n}(\theta\!+\!h)\!+\! \tilde{U} _{2,n}^\top(\!-\!h\!-\!\theta))}_{=2\tilde{U} _{2,n}(\theta\!+\!h)}C x_t(\theta)\dth \\
    \quad + \displaystyle \int_{-h}^0\! \int_{-h}^0 x_t^\top  (\theta_1)C^\top  \tilde{U}^\prime_{2,n}(\theta_2-\theta_1)Cx_t(\theta_2)\dth_1\dth_2,\\
    \quad - \displaystyle \int_{-h}^0\! \int_{-h}^0 x_t^\top  (\theta_1)C^\top  \tilde{U}^\prime_{2,n}(\theta_2-\theta_1)Cx_t(\theta_2)\dth_1\dth_2.
  \end{array}
\end{equation*}
We first notice that the two last terms of the previous expression are opposite and thus sum up to zero. Moreover, Property \ref{Prop:U2nSym} given in Appendix~\ref{app2} helps reducing the expression of the terms that depend on $\tilde{U} _{2,n}$. Re-ordering the previous expression yields 
\begin{equation}\label{eq:dV3}
   \dot{ \mathcal{V}}_n(x_t) \!=\! -\frac{1}{h} \!\displaystyle\int_{-h}^0 
    \begin{bsmallmatrix}
    x_t(0)\\
    hCx_t(\theta)\\
    Cx_t(-h)\\
    \end{bsmallmatrix}^\top\!\! \Psi_n(\theta) \!
    \begin{bsmallmatrix}
    x_t(0)\\
    hCx_t(\theta)\\
    Cx_t(-h)\\
    \end{bsmallmatrix}
    \dth,
\end{equation}
where $\Psi_n$ is given in~\eqref{eq:psin} and concludes the proof.
\end{proof}

To summarize the previous developments, a candidate functional $\mathcal{V}_n$ has been built to fulfill the structure given in~\eqref{thm1:defV} and to follow the time-derivative of the complete functional that is defined for any symmetric positive definite matrices $(W_1,W_2,W_3)$.

\subsection{Necessary condition for LMI condition~\eqref{eq:CSlmi1}}\label{sec:35}

As an extension to Lemma~\ref{lem:CN2n}, the next developments aim at understanding how the positive definiteness condition on the complete functional~$\mathcal{V}$ can be transferred to the candidate functional~$\mathcal{V}_n$ and then to LMI~\eqref{eq:CSlmi1} at any order. This is formulated in the next lemma. 
\begin{lemma}\label{cor:CN1}
If the trivial solution of system~\eqref{eq:tds} is~GES, then condition $\Phi_{n}^+\succ 0$ holds with matrices $(\Pn,S,R)$ given by~\eqref{lem:VnDef:Pn} for all $n$ in $\mathbb{N}^\ast$.
\end{lemma}
\begin{proof}
Assume that system~\eqref{eq:tds} is GES. According to Lemma~\ref{lem:CN2n}, the complete functional $\mathcal{V}$ in~\eqref{eq:V0} verifies inequality~\eqref{eq:Vnineq1} for a sufficiently small scalar $\varepsilon>0$ for all $\varphi$ in $\mathcal{C}_{pw}(-h,0;\mathbb{R}^{n_x})$.
In particular, consider $\varphi$ as a function expressed as follows
\begin{equation}\label{eq:varphi}
\varphi(\theta) =
\begin{bsmallmatrix}
\displaystyle\delta_0(\theta)I_{n_x}&
\displaystyle \frac{1}{h} C^{\dagger} f_n^\top(\theta) \mathcal{I}_n&\displaystyle
\delta_{-h}(\theta)C^{\dagger}
\end{bsmallmatrix}
\begin{bsmallmatrix}
    \varphi_0\\ 
 \varphi_n\\
\varphi_h\end{bsmallmatrix},
\end{equation}
with $C^{\dagger}=C^\top(CC^\top)^{-1}$ in $\mathbb R^{n_x\times n_z}$ is the right pseudo-inverse of $C$ and  with 
$$
f_n(\theta)=\ell_n(\theta)-\ell_n(1)\delta_0(\theta)-\ell_n(0)\delta_{-h}(\theta).
$$
In this formulation, vector $\begin{bsmallmatrix} 
\varphi_0\\ 
\varphi_n\\
\varphi_h\end{bsmallmatrix}$ in $\mathbb{R}^{n_x+n_z(n+1)}$ is arbitrary and $\delta_{\theta_0}$ is zero everywhere except at $\theta_0$, where it equals $1$, i.e. 
\begin{equation*}
    \delta_{\theta_0}(\theta)=\left\{\begin{array}{ll}1& \text{ for } \theta=\theta_0,\\0 & \text{ otherwise}.\end{array}\right.
\end{equation*} 
Note that such a function $\varphi$ has been selected so that $\varphi$ at the boundary of $[-h,0]$ is given by $\varphi(0)= \varphi_0$, $\varphi(-h)= C^\dagger \varphi_{h}$, and in the interval $(-h,0)$ is given by $\varphi(\theta)= \frac{1}{h}C^\dagger \ell_n(\theta)\mathcal{I}_n \varphi_n$, for all $\theta$ in $(-h,0)$., which is a polynomial function of $\theta$ of degree $n-1$.\newline
Re-injecting the particular function $\varphi$ into the definition of $\mathcal{V}$ in~\eqref{eq:V0}, the orthogonality of the Legendre polynomials decomposition~\eqref{eq:Untilde0} ensures that
\begin{equation}\label{eq:V3}
\begin{array}{l}
    \mathcal{V}(\varphi) \!=\! \varphi_0 U(0)\varphi_0 \!+\! \displaystyle 2\varphi_0^\top\!\!\underbrace{\int_{-h}^0\!\!\!\! U(\theta\!+\!h)B \ell_n^\top\!\left(\!\frac{\theta\!+\!h}{h}\!\right)\! \dth \frac{\mathcal{I}_n}{h}}_{\mathbf{U}_{1,n}} \varphi_n \\
    \!\! +\displaystyle\varphi_n^\top\!\underbrace{\frac{\mathcal{I}_n}{h}\! \!\!\!\overset{\ 0\ \ 0}{\underset{-h\ -h}{\iint}}\!\!\!\ell_n\!\!\left(\!\frac{\theta_1\!+\!h}{h}\!\right)\! B^\top U(\theta_2\!-\!\theta_1)B\ell_n^\top\!\!\left(\!\frac{\theta_2\!+\!h}{h}\!\right)\!\dth_{\!1}\dth_{\!2}\!\frac{\mathcal{I}_n}{h}}_{\mathbf{T}_{n}}\!\varphi_n\\
 \displaystyle+ \int_{-h}^0x^\top _t(\theta) C^\top(W_2+(\theta+h)W_3)Cx_t(\theta)\dth
    = \mathcal{V}_n(\varphi),
\end{array}
\end{equation}
where we recognize the functional $\mathcal{V}_n$ built with matrices $(\Pn,S,R)$ in~\eqref{lem:VnDef:Pn}. Hence, applying the Bessel-Legendre equality case~\eqref{lem:eqBessel} in Lemma~\ref{lem:Bessel} leads to
\begin{equation}\label{eq:v-}
 \mathcal{V}(\varphi) \!- \!\int_{-h}^0\!\!\!\varphi^\top\!(\theta) C^\top\!(\theta\!+\!h)W_2C\varphi(\theta)\dth\! =\! \begin{bsmallmatrix}
    \varphi_0\\ 
 \varphi_n\\
\end{bsmallmatrix}^\top \!\Phi_n^+ \begin{bsmallmatrix}
    \varphi_0\\ 
 \varphi_n\\
\end{bsmallmatrix}.
\end{equation}
Moreover, re-injecting \eqref{eq:varphi} into the inequality~\eqref{eq:Vnineq1} yields
\begin{equation}\label{ineq:v-}\begin{array}{r}
 \displaystyle\mathcal{V}(\varphi) \!- \!\int_{-h}^0\!\!\!\varphi^\top\!(\theta) C^\top\!(\theta\!+\!h)W_2C\varphi(\theta)\dth\qquad \qquad \\
 \!\geq\! \eps \left(\norm{\varphi_0}^2 \!\!+\! \varphi_n^\top\mathcal{I}_n\varphi_n\right) \!\geq\! \eps \norm{\begin{smallmatrix}\varphi_0\\ \varphi_n\end{smallmatrix}}^2\!\!.
\end{array}\end{equation}
Altogether, \eqref{eq:v-} and \eqref{ineq:v-} ensure  $\begin{bsmallmatrix}
\varphi_0\\ 
\varphi_n\\
\end{bsmallmatrix}^{\!\top}\! \Phi_n^+ \begin{bsmallmatrix}
    \varphi_0\\ 
 \varphi_n\\
\end{bsmallmatrix}
\!\geq\! \eps\left|\begin{smallmatrix}
    \varphi_0\\ 
 \varphi_n\\
\end{smallmatrix}\right|^2$, 
 for any vector $\left[\begin{smallmatrix}
    \varphi_0\\ 
 \varphi_n\end{smallmatrix} \right]$ in $\mathbb R^{n_x+nn_z}$, which guarantees the positive definiteness of matrix $\Phi_n^+$.
\end{proof}

\subsection{Existence of an order for having LMI condition~\eqref{eq:CSlmi2}}

As an extension to Lemma~\ref{lem:Vn}, the next developments use convergence arguments to asymptotically ensure the negative definiteness of the time-derivative of $\mathcal{V}_n$ along the trajectories of system~\eqref{eq:tds}. The following lemma guarantees the satisfaction of LMI~\eqref{eq:CSlmi2}, for sufficiently large orders.
\begin{lemma}\label{cor:CN2}
There exists an order $N^\ast$ in $\mathbb{N}^\ast$ such that LMI condition $\Phi_{n}^-\prec 0$ holds with matrices $(\Pn,S,R)$ given by~\eqref{lem:VnDef:Pn} for all $n \geq N^\ast$.
\end{lemma}
\begin{proof}
As a first step to the proof, let emphasize the structure of matrix $\Psi_n$ in~\eqref{eq:psin}. Such a matrix can be decomposed as the sum of a block diagonal positive definite matrices that is independent of $n$ and of a matrix whose entries are all expressed using the approximation errors $\tilde{U}_{n}$ and $\tilde{U}^\prime_{n}$, which can be made uniformly arbitrarily small in light of convergence properties in Lemma~\ref{lem:convergence}.\\
Therefore, we will use the following equivalence results. For any matrix $X$ in $\mathbb R^{p\times q}$ such that $X^\top X\preceq |X|^2 I_q$, then inequality $\begin{bsmallmatrix}
|X| I_p& X\\
X^\top & |X| I_q
\end{bsmallmatrix}\succeq0$ holds. Using this inequality, the following lower bounds of $\Psi_n$ is derived
\begin{equation*}
\Psi_n(\theta)\succeq \begin{bsmallmatrix}
    \mu_{1,n} I_{n_x} & 0 & 0\\
    0 & \mu_{2,n} I_{n_z} & 0\\
    0 & 0 & \mu_{3,n} I_{n_z}
    \end{bsmallmatrix},
\end{equation*}
with 
\begin{equation}\label{cor:defmu}
\begin{array}{rcl}
\mu_{1,n} \!&=&\! \sigmaD{W_1} \!-\!(1\!+\!|A|\!+\!2|C|)\bar{u}_{1,n}\!-\!\bar{u}_{2,n}\!-\!|C|\bar{u}_{3,n},\\ 
\mu_{2,n}  \!&=&\! \frac{1}{h}\sigmaD{W_2} \!-\!(|A|\!+\!\norm{B})\bar{u}_{1,n}\!-\!\bar{u}_{2,n}\!-\!(1\!+\!|C|)\bar{u}_{3,n},\\
\mu_{3,n} \!&=&\! \sigmaD{W_3} \!-\!(1\!+\!\norm{B})\bar{u}_{1,n}\!-\!\bar{u}_{3,n},
\end{array}
\end{equation}
where $\bar{u}_{1,n}$, $\bar{u}_{2,n}$ and $\bar{u}_{3,n}$ are the upper bounds of the approximation errors of the polynomial approximation given by~\eqref{eq:Wang},~\eqref{eq:Wang2} and~\eqref{eq:Wangbar}, respectively.\\
Then, Corollary~\ref{cor:bounds} ensures that all the negative terms can be made arbitrarily small as $n$ increases, while the first terms are positive and independent of $n$. In particular, there exist $N^\ast$ in $\mathbb{N}^\ast$ and a positive scalar $\eps=\mathrm{min}(\mu_{1,n},h\mu_{2,n},\mu_{3,n})>0$ such that the functional $\mathcal{V}_n$ associated to matrices $(\Pn,S,R)$ given by~\eqref{thm1:defV} satisfies
\begin{equation}\label{eq:Vnineq2}
    \dot{\mathcal{V}}_n(\varphi) \leq - \displaystyle\!\varepsilon \left(\norm{\varphi(0)}^2 \!+\!\int_{-h}^0|C\varphi(\theta)|^2\dth\!+\!\norm{C\varphi(-h)}^2\right)\!,
\end{equation}
for all $\varphi$ in $\mathcal{C}_{pw}(-h,0;\mathbb{R}^{n_x})$ and $n\geq N^\ast$.\newline
As a second step, we show that this inequality leads to the satisfaction of LMI condition~\eqref{eq:CSlmi2}. Indeed, consider $\varphi$ as the function expressed in~\eqref{eq:varphi}, which is a polynomial of degree $n-1$ over the interval $(-h,0)$.
Re-injecting expression~\eqref{eq:varphi} into the definition the time-derivative of the functional $\mathcal{V}_n$ in \eqref{eq:dotV} yields
\begin{equation}
 \dot{\mathcal{V}}_n(\varphi)=\begin{bsmallmatrix}
    \varphi_0\\ 
 \varphi_n\\
\varphi_h\end{bsmallmatrix}^\top\! \Phi_n^- \begin{bsmallmatrix}
    \varphi_0\\ 
 \varphi_n\\
\varphi_h\end{bsmallmatrix},
\end{equation}
where we have used the Bessel-Legendre equality case~\eqref{lem:eqBessel} of Lemma~\ref{lem:Bessel}, since $\varphi$ is a polynomial of degree $n-1$ over $(-h,0)$.
Similarly, re-injecting \eqref{eq:varphi} into the right-hand part of inequality~\eqref{eq:Vnineq2} leads to
\begin{equation}
    \dot{\mathcal{V}}_n(\varphi) \!\leq\! -\eps\! \left(\!\norm{\varphi_0}^2\!+\! \varphi_n^\top\mathcal{I}_n\varphi_n\!+\! \norm{\varphi_h}^2\!\right) \!\leq\! -\eps \norm{\begin{smallmatrix}\varphi_0\\ \varphi_n\\\varphi_h\end{smallmatrix}}^2.
\end{equation}
Hence, the following inequality holds 
\begin{equation*}
 \begin{bsmallmatrix}
    \varphi_0\\ 
 \varphi_n\\
\varphi_h\end{bsmallmatrix}^\top \Phi_n^- \begin{bsmallmatrix}
    \varphi_0\\ 
 \varphi_n\\
\varphi_h\end{bsmallmatrix}
\leq  
-\eps\left|\begin{smallmatrix}
    \varphi_0\\ 
 \varphi_n\\
  \varphi_h\\
\end{smallmatrix}\right|^2, \forall\begin{bsmallmatrix}
    \varphi_0\\ 
 \varphi_n\\
\varphi_h\end{bsmallmatrix}\in\mathbb R^{n_x+(n+1)n_z},
\end{equation*}
which ensures that matrix $\Phi_n^-$ is negative definite and concludes the proof.
\end{proof}
\vspace{-0.3cm}

\subsection{Necessary LMI conditions}

\vspace{-0.3cm}
The final step consists in demonstrating that the previous corollaries impose the satisfaction of the LMI in Theorem~\ref{thm:CSlmi}, which corresponds to the following statement.

\begin{theorem}\label{thm:CN1}
If the trivial solution of system~\eqref{eq:tds} is~GES, then then there exist an order ${N^\ast}$ in $\mathbb{N}^{\ast}$ and matrices $(\mathbf{P}_{N^\ast},R,S)$ in $\mathbb{S}^{n_x+{N^\ast}n_z}\times\mathbb{S}^{n_z}_+\times\mathbb{S}^{n_z}_+$ such that inequalities $\Phi_{N^\ast}^+ \succ 0$ and $\Phi_{N^\ast}^- \prec 0$ in~\eqref{eq:CSlmi} hold.
\end{theorem}
\begin{proof}
Consider the functional $\mathcal V_n$ given by~\eqref{thm1:defV} associated to matrices $(\Pn,S,R)$ in $\mathbb{S}^{n_x+n n_z}\times \mathbb{S}^{n_z}_+\times\mathbb{S}^{n_z}_+$ given by~\eqref{lem:VnDef:Pn}. Recall that these matrices have been built for given matrices $(W_1,W_2,W_3)$ in $\mathbb{S}^{n_x}_+\times \mathbb{S}^{n_z}_+\times \mathbb{S}^{n_z}_+$ with the idea to propose an approximated reconstruction of the complete functional $\mathcal{V}$. We have then shown in Lemma~\ref{cor:CN1} that having system~\eqref{eq:tds} GES guarantees that $\Phi_{n}^+\succ 0$ holds, for any order $n$ in $\mathbb{N}^\ast$. Applying Lemma~\ref{cor:CN2}, we have also ensured the existence of an order $N^\ast$ such that $\Phi_{N^\ast}^-\prec 0$ holds. These two results close the proof.
\end{proof}

To sum up, this section has been dedicated to the proof of the converse side of Theorem~2. The next section aims at providing an estimation of the order $N^\ast$ for which this converse statement holds.

\section{Necessity of LMI stability conditions: Estimation of the order $N^\ast$}\label{sec3}

This section aims at providing an estimation of the order $N^\ast$ for which the sufficient LMI conditions of Theorem~\ref{thm:CSlmi} are necessarily true if the system is assumed to be GES.

\begin{theorem}\label{thm:CN2}
If the trivial solution of system~\eqref{eq:tds} is GES, then there exist matrices $(\mathbf{P}_{\!N^\ast},R,S)$ in $\mathbb{S}^{n_x\!+\!{N^\ast}\!n_z}~\times~\mathbb{S}^{n_z}_+~\times~\mathbb{S}^{n_z}_+$ such that inequalities $\Phi_{N^\ast}^+ \succ 0$ and $\Phi_{N^\ast}^- \prec 0$ in~\eqref{eq:CSlmi} hold with 
\begin{equation}\label{eq:Nast}
N^\ast \!\!=\! 5 \!+\! \left\lceil\! \left((\textstyle\frac{4}{1+h^2}\!+\!\norm{A}\!+\!\norm{B})\varrho_1\!+\!\varrho_2\!+\!2\varrho_3\right)^{\!2}\!\left(1\!+\!h^2\right)^2\!\right\rceil\!,
\end{equation}
where parameters $\varrho_1,\varrho_2,\varrho_3$ are defined by~\eqref{eq:r1},~\eqref{eq:r2},~\eqref{eq:r3}, but are recalled here for consistency
\begin{equation*}
\begin{array}{lcl}
    \varrho_1  &=& \sqrt{n_x} \left(\frac{\pi}{2}\right)^\frac{3}{2}e^{h\norm{\mathcal{M}}}\norm{\mathcal{M}^2\mathcal{N}^{-1}}h^2 \norm{B},\\
    \varrho_2  &=& \frac{1}{2}\sqrt{n_x} \left(\frac{\pi}{2}\right)^\frac{3}{2}e^{h\norm{\mathcal{M}}}\norm{\mathcal{M}^4\mathcal{N}^{-1}}h^3\norm{B},\\
    \varrho_3  &=& \sqrt{2\pi} ( 1 +  \sqrt{n_x}\pi h e^{h\norm{\mathcal{M}}}\norm{\mathcal{M}^2\mathcal{N}^{-1}}) h \norm{B}^2.
\end{array}
\end{equation*}
\end{theorem}

\begin{proof}
The key step for the estimation of $N^\ast$ appears in the first part of the proof of Lemma~\ref{cor:CN2}, more particularly in equations~\eqref{cor:defmu}, that are recalled here for the sake of readability (with $|C|=1$),
\begin{equation}\label{proof:defmu}
\begin{array}{rclcl}
0\!&<&\!\mu_{1,n} \!&=&\! \sigmaD{W_1} \!-\!(3\!+\!|A|)\bar{u}_{1,n}\!-\!\bar{u}_{2,n}\!-\!\bar{u}_{3,n},\\ 
0\!&<&\!\mu_{2,n} \!&=&\! \frac{1}{h} \sigmaD{W_2} \!-\!(|A|\!+\!\norm{B})\bar{u}_{1,n}\!-\!\bar{u}_{2,n}\!-\!2\bar{u}_{3,n},\\
0\!&<&\!\mu_{3,n}  \!&=&\! \sigmaD{W_3} \!-\!(1\!+\!\norm{B})\bar{u}_{1,n}\!-\!\bar{u}_{3,n}.
\end{array}
\end{equation}
As matrices $(W_1,W_2,W_3)$ are arbitrary, it is possible to select for all $\lambda>0$
\begin{equation}
W_1=\lambda \eta_1 I_{n_x}, \quad hW_2= \lambda\eta_2 I_{n_z},\quad W_3=\lambda\eta_3I_{n_z},
\end{equation}
where $\eta_1$, $\eta_2$ and  $\eta_3>0$ are positive scalars such that $\eta_1+\eta_2+\eta_3=1$. This selection makes that 
$$|W|=|W_1+hW_2+W_3|=\lambda.$$ Therefore, the upper bounds of the approximation errors given in~\eqref{eq:Wang},~\eqref{eq:Wang2} and~\eqref{eq:Wangbar} verify for all $n\geq6$, 
\begin{equation*}
\begin{array}{lcll}
  \bar{u}_{1,n}&=&\displaystyle \frac{\varrho_1}{\sqrt{n-3}} \lambda \leq \frac{\varrho_1}{\sqrt{n-5}} \lambda,\\
   \bar{u}_{2,n}&=&\displaystyle \frac{\varrho_2}{\sqrt{n-5}}\lambda,\\
    \bar{u}_{3,n}&= &\displaystyle \frac{\varrho_3}{\sqrt{n-3}} \lambda \!\leq\! \frac{\varrho_3}{\sqrt{n-5}} \lambda.
\end{array}
\end{equation*}
The objective is now to evaluate the necessary condition obtained in the previous section, more particularly the ones arising from inequalities~\eqref{proof:defmu}. We select then $\eta_1$, $\eta_2$ and $\eta_3$ as follows
\begin{equation*}
\begin{array}{rcl}
    \eta_1 &\geq& \displaystyle \frac{\varrho_1(3+\norm{A}) + \varrho_2 + \varrho_3}{\sqrt{n-5}},\\
    \eta_2 &\geq& \displaystyle \frac{\varrho_1(\norm{A}+\norm{B}) + \varrho_2 + 2\varrho_3}{\sqrt{n-5}}h^2,\\
    \eta_3 &\geq& \displaystyle \frac{\varrho_1(1+\norm{B}) + \varrho_3}{\sqrt{n-5}}.
\end{array}
\end{equation*}
Recalling that $\eta_1+\eta_3+\eta_3=1$, we obtain
\begin{equation*}
\frac{4\varrho_1\!+\!\big((\norm{A}\!+\!\norm{B})\varrho_1\!+\!\varrho_2\!+\!2\varrho_3\big)(1\!+\!h^2)}{\sqrt{n-5}}\leq 1, 
\end{equation*}
which yields condition \eqref{eq:Nast}.
\end{proof}

\begin{remark}\label{rem:Nast}
It is worth noticing that the minimal order $N^\ast$ given by~\eqref{eq:Nast} is an over estimate of the necessary order. Indeed, to get such estimate, several over-bounding approximations have been performed, leading to a conservative estimation. Nevertheless, this is a first solution to the problem of necessary and sufficient LMI conditions for the stability of time-delay systems.  
Several improvements to get a more accurate estimation is let to future works. One of the possible improvements would consider, for instance, more accurate estimation of  $\varrho_1$ and $\varrho_2$ which would greatly reduce the estimation by using the fact that $U$ is infinitely continuous on $[-h,0)$. Another possibility to improve our estimation of $N^\ast$ is to compute the minimal order directly from inequality~\eqref{eq:psin}.
\end{remark}

\section{Numerical results}\label{sec:ex}

The numerical application of Theorems~\ref{thm:CSlmi} and~\ref{thm:CN2} is commented and illustrated on the following academic examples corresponding to \eqref{eq:tds} with
\begin{example}\label{ex1}
 $A=1$ and $A_d=-2$.
\end{example}

\begin{example}\label{ex2}
$A=\begin{bsmallmatrix}0&0\\0&0\end{bsmallmatrix}$ and $A_d=\begin{bsmallmatrix}-1&0.2\\-0.1&0\end{bsmallmatrix}$.
\end{example}

\begin{example}\label{ex3}
$A\!=\!\begin{bsmallmatrix}0&0&1&0\\0&0&0&1\\\!-10-\lambda&10&0&0\\5&-15&0&-\frac{1}{4}\end{bsmallmatrix}$ and $A_d\!=\!\begin{bsmallmatrix}0&0&0&0\\0&0&0&0\\\lambda&0&0&0\\0&0&0&0\end{bsmallmatrix}$.
\end{example}

First, we recall the efficiency of the sufficient stability condition of Theorem~\ref{thm:CSlmi}. For Examples~\ref{ex1} and~\ref{ex2}, Table~\ref{tab:CS} reports the maximal allowable delay, for which LMI~\eqref{eq:CSlmi} at order $n=1,2,3$ are satisfied.  Figure~\ref{fig1} shows the stability region for Example~\ref{ex3} achieved by the same condition for several orders. One can see that in this table and in this figure the efficiency of Theorem~\ref{thm:CSlmi} even for very low orders $n$ to provide an inner approximation of the stability regions.

\begin{table}[!t]
    \caption{Sufficiency: maximal allowable delay $h$ which satisfies~\eqref{eq:CSlmi}.}
    \centering
    \begin{tabular}{|c|c|c|c|c|}
    \hline
        & $n=1$ & $n=2$ & $n=3$ & Expected\\
    \hline
        Ex.~\ref{ex1} & $0.577$ & $0.604$ & $0.604$ & $\frac{\mathrm{atan}\sqrt{3}}{\sqrt{3}}\simeq0.604$ \\
        Ex.~\ref{ex2} & $-$ & $1.600$ & $ 1.603$ & $1.603$\\
    \hline
    \end{tabular}
    \label{tab:CS}
\end{table}

Thanks to Theorem~\ref{thm:CN2}, these LMI conditions of stability are proven to be also necessary. Therefore, the inner approximation mentioned above is supposed to converge toward the expected regions of stability as the order $n$ increases. Table~\ref{tab:CN} reports the estimated order $N^\ast$ for which the LMI conditions become necessary. Similarly, Figure~\ref{fig2} shows the values of $N^\ast$  given by~\eqref{eq:Nast} for various pairs of $(\lambda,h)$ in $[0.1,10]\times [0,3]$. Clearly the values of $N^\ast$ computed here are too large to propose tractable test of instability in comparison to~\cite{Mondie2021}. Nonetheless, Theorem~\ref{thm:CN2} provides a theoretical estimation of $N^\ast$, from which these sufficient LMI conditions become necessary. 

This estimation is a by-product of our main result that has not been optimized in this paper. Some guidelines to improve this estimation have been suggested in Remark~\ref{rem:Nast}.

\begin{table}[!t]
    \centering
    \caption{Necessity: estimated order $N^\ast$ given by~\eqref{eq:Nast}.}
    \begin{tabular}{|c|c|c|c|c|}
    \hline
        & $h=0.1$ & $h=0.5$ & $h=1$ & $h=2$\\
    \hline
        Ex.~\ref{ex1} & $75$ & $10^8$ & $10^{10}$ & $10^{13}$ \\
        Ex.~\ref{ex2} & $6$ & $537$ & $10^7$ & $10^{10}$\\
    \hline
    \end{tabular}
    \label{tab:CN}
\end{table}

Interestingly Figure~\ref{fig2} shows that increasing both $\lambda$ and $h$ makes that $N^\ast$ also increases very fast and reach very large values. Indeed, formula~\eqref{eq:Nast} shows that the order $N^\ast$ grows as the delay $h$ or the norm $\norm{B}$ increase, respectively in $h^8e^{2h\norm{M}}$ and $\norm{B}^4$. It is also worth noticing that when the parameter $\lambda$ and $h$ are getting closer to the black lines, the estimation $N^\ast$ increases even more faster. These black lines correspond to the situation where some roots of the time-delay system cross the imaginary axis. This makes sense because if characteristic roots are approaching the imaginary axis, then matrix $\mathcal{N}$ tends to a singular matrix and $\norm{\mathcal{N}^{-1}}$ tends to infinity. Hence, the upper bounds $\varrho_1,\varrho_2,\varrho_3$ also tend to infinity.

All together, this ascertainment can be correlated with Figure~\ref{fig1} and may explain why some stable regions are difficult to reach with LMI~\eqref{eq:CSlmi}, especially for low orders.

\begin{figure}
    \centering
    \begin{subfigure}{0.99\linewidth}
    \centering
	\includegraphics[width=8cm]{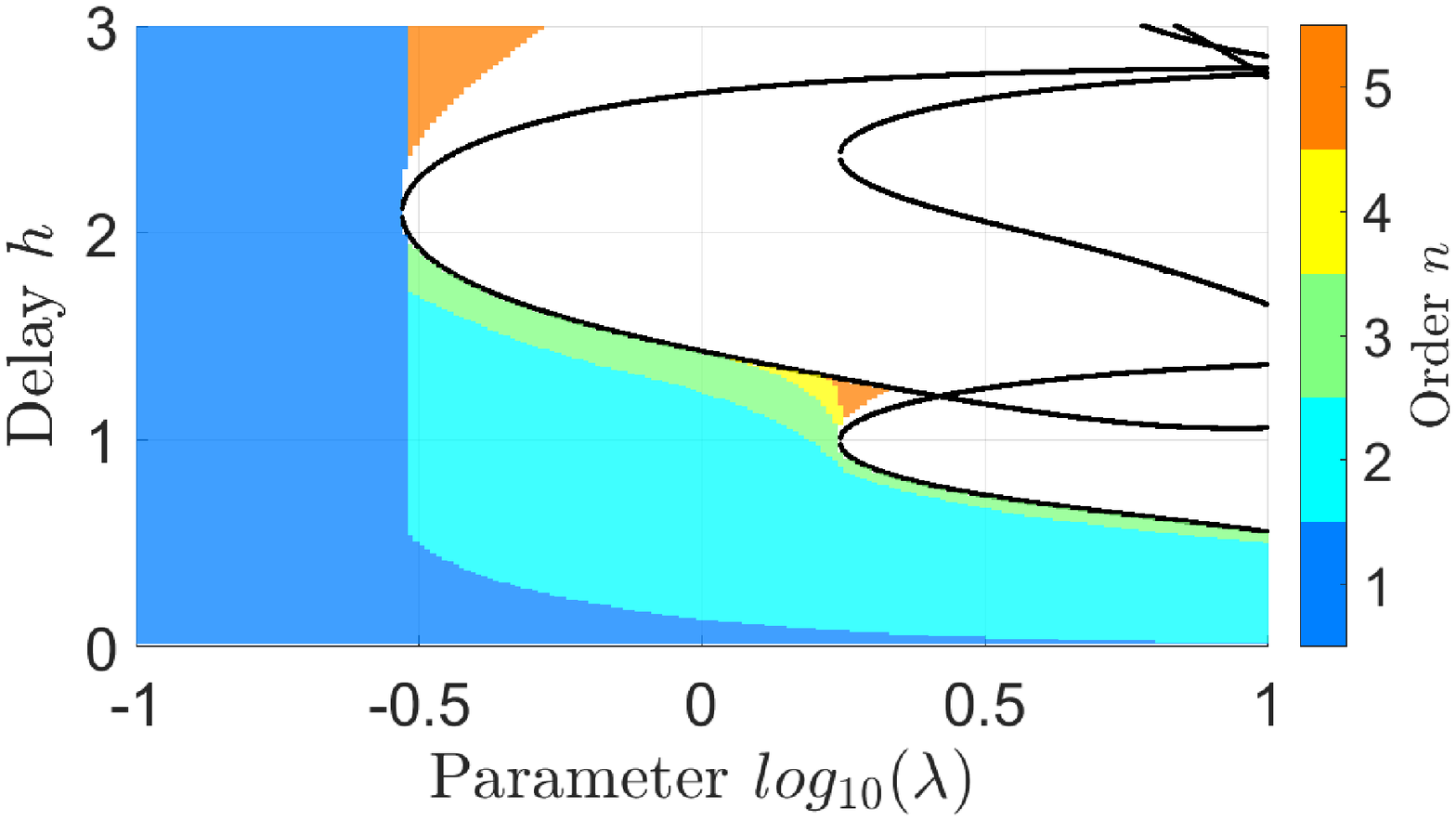}
	\caption{Sufficiency: stability areas in the plan $(\lambda,h)$ given by~\eqref{eq:CSlmi} for $n\in\{1,2,3,4,5\}$.}
	\label{fig1}
	\end{subfigure}\\
	\begin{subfigure}{0.99\linewidth}
    \centering
	\includegraphics[width=8cm]{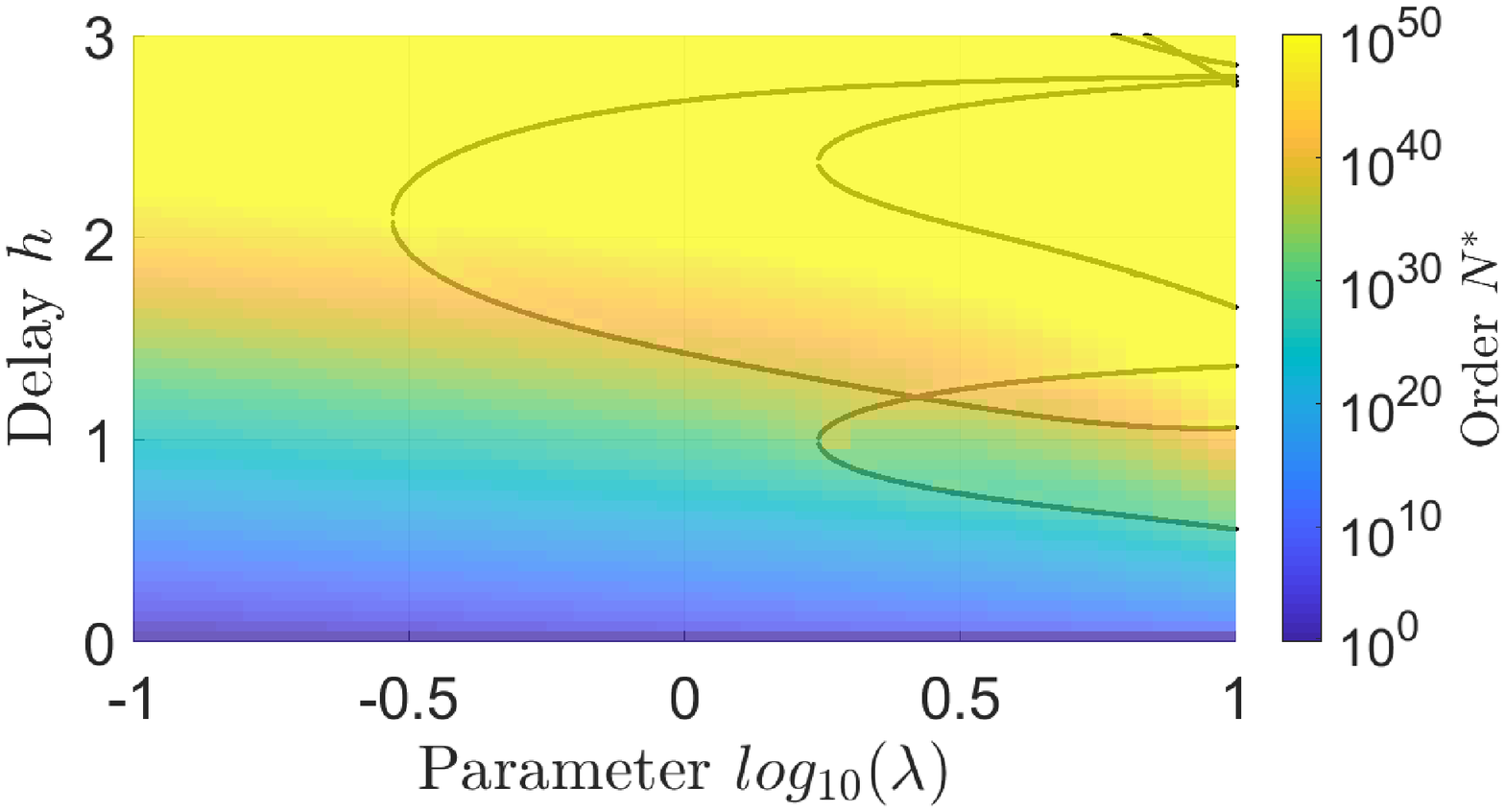}
	\caption{Necessity: order $N^\ast$ in the plan $(\lambda,h)$ given by~\eqref{eq:Nast}.}
	\label{fig2}
	\end{subfigure}
	\caption{Example~\ref{ex3}.}
\end{figure}

\section{Conclusions}

This paper studied the convergence of sufficient LMI conditions for the stability analysis of time-delay systems, based on the Bessel-Legendre inequalities. While the framework of Bessel-Legendre already showed their relevance regarding the hierarchical structure of these LMI, i.e. increasing their order can only reduce their conservatism, the main contribution of this paper demonstrates that this framework also offers an asymptotically necessary condition of stability for time-delay systems. In other words, it is now proven that if a time-delay system is stable, then there exists an order $N^\ast$ such that these LMI conditions are verified at least at this order. A numerical estimation of this necessary order has also been provided, ensuring that if these LMI conditions are not verified at this order, then the system is proven to be unstable. To summarize, these LMI conditions arising from the Bessel-Legendre inequality are sufficient and asymptotically necessary. 

The estimation of the order $N^\ast$ provided in this paper can be numerically very large, even on some simple examples. This can be seen as a major drawback of the contribution. Nevertheless, the main theoretical achievement of the paper remains the proof of convergence of these LMI conditions.\\
Providing more accurate estimations is let to future direction of research. A generalization to other approximation or discretization methods or the investigation of combined procedures could also be considered.

\appendix

\section{Preliminaries on Legendre Polynomials for the technical proofs}\label{app:leg}

In this section, several properties of Legendre polynomials are presented to help in the technical developments of this paper.

\begin{property}\label{prop:legderivative}
The Legendre polynomials verify the following properties
\begin{itemize}
\item[(i)] \textbf{Orthogonality:} The following equality holds
\begin{equation}\label{eq:ortho}
    \int_0^1 \!\!\ell_n(\theta)\ell_n^\top(\theta)\dth = \mathcal{I}_n^{-1} \in\mathbb{R}^{nn_z\times nn_z},
\end{equation}
where matrix $\mathcal{I}_n=\mathcal{I}_n^{n_z}=\mathcal{I}_n^1\otimes I_{n_z}$ is given by~\eqref{eq:In}.
\item[(ii)] \textbf{Point-wise values:} The Legendre polynomials are evaluated point wisely by 
\begin{equation}\label{eq:legeval}
    l_k(0) = (-1)^k,\; l_k(1) = 1,\; l_k'(1)= k(k+1),\; \forall k\in\mathbb{N}.
\end{equation}
\item[(iii)] \textbf{Evenness and oddness:} For any $\theta_1$ and $\theta_2$ in $[-h,0]$ and  $n$ in $\mathbb{N}^{\ast}$, the following equation holds
\begin{equation}\label{eq:evenlk}
\ell_n^\top\!\left(\!\frac{h\!-\!\theta_1}{2h}\!\right)\! \mathcal{I}_n\ell_n\!\left(\!\frac{h\!+\!\theta_2}{2h}\!\right)\!=\!\ell_n^\top\!\left(\!\frac{h\!-\!\theta_2}{2h}\!\right)\!\mathcal{I}_n\ell_n\!\left(\!\frac{h\!+\!\theta_1}{2h}\!\right)\!.
\end{equation}
\item[(iv)] \textbf{Bound:} For any $\theta$ in $[0,1]$, the Legendre polynomials verify the following inequalities
\begin{align}
\norm{l_k(\theta)}&\leq 1,\qquad\qquad\qquad\quad  \forall k\in\mathbb{N}, \label{eq:boundlk}\\
\norm{l_k(\theta)}&\leq\frac{1}{2}\sqrt{\frac{\pi}{2k\theta(1-\theta)}},\quad \ \forall k\in\mathbb{N}^\ast.\label{eq:boundlk2}
\end{align}
\item[(v)] \textbf{Differentiation:} For any $\theta$ in $[0,1]$, the Legendre polynomials verify the following differentiation rule
\begin{equation}\label{eq:derlk}
l_k(\theta)= \frac{1}{2(2k+1)}\left(l_{k+1}^\prime(\theta)-l_{k-1}^\prime(\theta)\right), \; \forall k\in\mathbb{N}^\ast.
\end{equation}
\item[(vi)] \textbf{Bound on the derivative:} For any $\theta$ in $[0,1]$, the Legendre polynomials verify the following inequality
\begin{equation}\label{eq:boundderlk}
\norm{l_k^\prime(\theta)}\leq \norm{l_k^\prime(1)}=k(k+1),\quad \forall k\in\mathbb{N}.
\end{equation}
\item[(vii)] 
\textbf{Bonnet's recursion formula:} For any $\theta$ in $[0,1]$ and $k$ in $\mathbb{N}^\ast$, the Legendre polynomials satisfy
\begin{equation}\label{eq:reclk}
    (k+1) l_{k+1}(\theta) = (2k+1) l_1(\theta) l_k(\theta) - k l_{k-1}(\theta).
\end{equation}
\end{itemize}
\end{property} 

\begin{proof}
The properties of orthogonality~(i), special values~(ii)-(iii) and differentiation~(v) are provided in~\cite[Formula 22.2.11, 22.4.5 and 22.7.15]{Abramowitz1972}, respectively. For the boundedness property, especially~\eqref{eq:boundlk2}, the proof can be found in~\cite[Theorem~61]{Rainville1960}.
\end{proof}

\section{Symmetric property}\label{app2}

In the core of the paper, we have used the following property satisfied by the Legendre approximated error $\tilde{U}_{2,n}$ of function $B^\top UB$ on the interval $[-h,h]$.

\begin{property}\label{Prop:U2nSym}
The approximation error $\tilde{U}_{2,n}$ given by~\eqref{eq:Untilde0} and recalled below
\begin{equation*}
    \tilde{U}_{2,n}(\theta) = B^\top U(\theta) B - \mathbf{U}_{2,n} \ell_n\!\left(\frac{\theta+h}{2h}\right),\, \forall \theta\in[-h,h],
\end{equation*}
verifies 
$\tilde{U}_{2,n}(\theta)=\tilde{U}_{2,n}^\top(-\theta)$, for all $\theta\in[-h,0]$.
\end{property}

\begin{proof}
From the definitions of $\tilde U_{2,n}$ and $\mathbf{U}_{2,n}$, we have
\begin{equation*}
\begin{array}{lcl}
\tilde{U}_{2,n}^\top (-\theta)=  B^\top U^\top(-\theta) B  \\
- \displaystyle \frac{1}{2h}\left(\!\int_{-h}^h \ell_n^\top\!\left(\!\frac{h\!-\!\theta}{2h}\!\right) \mathcal{I}_n\ell_n\!\left(\!\frac{h\!+\!\theta_1}{2h}\!\right) B^\top U^\top(\theta_1)B\dth_1\!\right)\!.
\end{array}
\end{equation*}
Then, property~\eqref{eq:evenlk} of the Legendre polynomials and since $\ell_n^\top\left(\frac{h-\theta_1}{2h}\right) \mathcal{I}_n\ell_n\!\left(\frac{h+\theta}{2h}\right)$ is an $n_z\times n_z$ matrix that is proportional to $I_{n_z}$, it commutes with $B^\top U^\top(\theta_1)B$, so that the previous expression writes
$$\begin{array}{lcl}
\tilde{U}_{2,n}^\top (-\theta)=  B^\top U^\top(-\theta) B  \\
- \displaystyle \frac{1}{2h}\left(\!\int_{-h}^h  B^\top U^\top(\theta_1)B \ell_n^\top\!\left(\!\frac{h\!-\!\theta_1}{2h}\!\right) \mathcal{I}_n\ell_n\!\left(\!\frac{h\!+\!\theta}{2h}\!\right)\dth_1\!\right)\!.
\end{array}
$$
Recalling that $U^\top(-\theta) =U(\theta)$ holds for all $\theta$ in $[-h,h]$ and performing the  change of variable $\theta_2=-\theta_1$, the previous expression becomes
$$\begin{array}{lcl}
\tilde{U}_{2,n}^\top (-\theta)=  B^\top U(\theta) B  \\
- \displaystyle \frac{1}{2h}\left(\int_{-h}^h  B^\top U(\theta_2)B \ell_n^\top\!\left(\frac{h\!+\!\theta_2}{2h}\right)\dth_2\right)  \mathcal{I}_n\ell_n\!\left(\frac{h\!+\!\theta}{2h}\right),
\end{array}
$$
which is the definition of $\tilde{U}_{2,n}(\theta)$.
\end{proof}

\section{Proof of Lemma~\ref{lem:Bessel}}\label{app0}

\begin{proof}
Consider a function $z$ in $\mathcal{L}_2(-h,0;\mathbb{R}^{n_z})$, a matrix $S$ in $\mathbb{S}^{n_z}$ and $h > 0$. Define the function $\tilde{z}_n$ by
\begin{equation}\label{eq:tildezn}
    \tilde{z}_n(\theta) \!=\! z(\theta) - \frac{1}{h} \ell_n^\top\left(\frac{\theta+h}{h}\right)\mathcal{I}_n\zeta_n(z),
\end{equation}
where $\zeta_n(z) = \int_{-h}^0\!\!\ell_n\left(\frac{\theta_1+h}{h}\right)z(\theta_1)\dth_1$.
The function $\tilde{z}_n$ in $\mathcal{L}_2(-h,0;\mathbb{R}^{n_z})$ represents the Legendre approximation error of function $z$ at order $n$.
The quantity $\int_{-h}^0\tilde{z}_n^\top(\theta) S \tilde{z}_n(\theta) \mathrm{d}\theta$ exists and the orthogonal property of the Legendre polynomials~\eqref{eq:In} yields
\begin{equation*}
    \begin{aligned}
        \int_{-h}^0\tilde{z}_n^\top(\theta) S \tilde{z}_n(\theta) \mathrm{d}\theta = &\int_{-h}^0 z^\top(\theta) S z(\theta) \mathrm{d}\theta \\
        &-\frac{1}{h}\zeta_n^\top(z) (\mathcal{I}_n^1\!\otimes\! S) \zeta_n(z),
    \end{aligned}
\end{equation*}
Clearly, if $z$ is a polynomial of order $n-1$, then $\tilde{z}_n(\theta)=0$ for all $\theta\in(-h,0)$ so that $\int_{-h}^0\tilde{z}_n^\top(\theta) S \tilde{z}_n(\theta) \mathrm{d}\theta=0$ , which ensures that the equality~\eqref{lem:eqBessel} holds. Otherwise, recalling that $\int_{-h}^0\tilde{z}_n^\top(\theta) S \tilde{z}_n(\theta) \mathrm{d}\theta > 0$ since $S\succ 0$, inequality~\eqref{lem:ineqBessel} holds and concludes the proof.
\end{proof}

\section{Proof of Lemma~\ref{lem:CN2n}}\label{app3}

\begin{proof}
First, let us introduce the functional
\begin{align}\label{eq:W}
    \mathcal{W}(\varphi) =& \mathcal{V}(\varphi) - \int_{-h}^0\varphi^\top(\theta) C^\top(\theta+h)W_2C\varphi(\theta)\dth\nonumber\\
    &- \varepsilon \int_{-h}^0 \begin{bsmallmatrix}\varphi(0)\\C\varphi(\theta)\end{bsmallmatrix}^\top  \begin{bsmallmatrix}\varphi(0)\\C\varphi(\theta)\end{bsmallmatrix}\dth,
\end{align}
where $\mathcal{V}$ is the complete Lyapunov-Krasovskii functional given by~\eqref{eq:V0}.
According to Lemma~\ref{lem:dV}, the time-derivative of $\mathcal{W}$ along the trajectories of system~\eqref{eq:tds} writes
\begin{equation}\label{eq:suppW}
\begin{array}{rcl}
    \dot{\mathcal{W}}(x_t) &=& - \begin{bsmallmatrix}x_t(0)\\Cx_t(-h)\end{bsmallmatrix}^\top \begin{bsmallmatrix} W_1+C^\top hW_2C & 0 \\ \ast &  W_3\end{bsmallmatrix} \begin{bsmallmatrix}\varphi(0)\\Cx_t(-h)\end{bsmallmatrix}\\
    && +\varepsilon\begin{bsmallmatrix}x_t(0)\\Cx_t(-h)\end{bsmallmatrix}^\top \begin{bsmallmatrix} \mathcal{H}(A)+C^\top C & B \\ \ast &  -I_{n_z}\end{bsmallmatrix} \begin{bsmallmatrix}x_t(0)\\Cx_t(-h)\end{bsmallmatrix}.
\end{array}
\end{equation}
Clearly, there exists a sufficiently small $\varepsilon>0$ such that $\dot{\mathcal{W}}(x_t)\leq 0$. Then, integrating~\eqref{eq:suppW} from $0$ to $\infty$ and assuming that system~\eqref{eq:tds} is GES yields $\mathcal{W}(x_0) \geq 0$, for any initial conditions $x_0=\varphi$ in $\mathcal{C}_{pw}(-h,0;\mathbb{R}^{n_z})$, which concludes the proof.
\end{proof}

\section{Proof of Lemma~\ref{lem:convergence}}\label{app1}

The next properties reflect the fact that $U$ has continuous and bounded second and fourth order derivatives on $[-h,0)$ (and $(0,h]$). Their bound are provided therein.
\begin{property}\label{prop:regul}
The Lyapunov matrix $U$ associated to $W$ defined in~\eqref{eq:U0} satisfies
\begin{align}
      \underset{\theta\in[-h,0)}{\mathrm{sup}}\!
      \norm{U^{(2)}(\theta)} \!=\!  \underset{\theta\in(0,h]}{\mathrm{sup}}\! \norm{U^{(2)}(\theta)} &\leq \rho \norm{W} ,\label{eq:Uregul1}
  \\
     \underset{\theta\in(0,h]}{\mathrm{sup}}\! \norm{U^{(4)}(\theta)} &\leq {\rho ^\prime \norm{W} } ,\label{eq:Uregul2}
\end{align}
with parameters $\rho,\rho^\prime$ given by
\begin{equation}\label{eq:rv}
    \rho  \!=\! \sqrt{n_x} e^{h\norm{\mathcal{M}}}\norm{\mathcal{M}^2\mathcal{N}^{-1}},\;
   \rho' \!=\! \sqrt{n_x} e^{h\norm{\mathcal{M}}}\norm{\mathcal{M}^4\mathcal{N}^{-1}},
\end{equation}
and matrices $\mathcal{M}$, $\mathcal{N}$ given by~\eqref{eq:MN0}.
\end{property}

\begin{proof}
Thanks to the equivalence of matrix norms, inequalities $\norm{M^\top}=\norm{M}\leq\norm{\mathrm{\mathrm{vec}}(M)}\leq\sqrt{p}\norm{M}$ hold, for any square $M$ of dimension $p$. Then, for all $\theta$ in $(0,h]$, we have
\begin{equation*}
\begin{aligned}
    \norm{U^{(k)}(\theta)} &\leq \left| \begin{bsmallmatrix}I_{n_x^2}&0\end{bsmallmatrix}e^{\theta\mathcal{M}}\mathcal{M}^k\mathcal{N}^{-1}\begin{bsmallmatrix}-\mathrm{vec}(W)\\0\end{bsmallmatrix}\right|,\\
    &\leq \norm{e^{\theta\mathcal{M}}}\norm{\mathcal{M}^k\mathcal{N}^{-1}}\norm{\mathrm{vec}(W)},\\
    &\leq \sqrt{n_x} \norm{e^{\theta\mathcal{M}}} \norm{\mathcal{M}^k\mathcal{N}^{-1}}\norm{W},
\end{aligned}
\end{equation*}
where $U^{(k)}$ stands for the $k^{th}$ derivatives of function $U$.
Moreover, recalling the definition of exponential matrices, i.e.  $e^{\theta\mathcal{M}}=\sum_{k=0}^\infty\frac{(\theta\mathcal{M})^k}{k!}$, an upper bound of $|e^{\theta \mathcal M}|$ is obtained as follows
\begin{equation*}
    \norm{e^{\theta\mathcal{M}}} \!\leq\! \sum_{k=0}^\infty \!\! \left| \! \frac{(\theta\mathcal{M})^k}{k!} \! \right| \!\leq\! \sum_{k=0}^\infty \!\! \frac{|\theta|^k|\mathcal{M}|^k}{k!} \!\leq\! \sum_{k=0}^\infty\!\! \frac{h^k |\mathcal{M}|^k}{k!} \!\!=\! e^{h\norm{\mathcal{M}}}\!,
\end{equation*}
which yields the results~\eqref{eq:Uregul1} and \eqref{eq:Uregul2}. 
\end{proof}

\subsection{Proof of item (i) of Lemma~\ref{lem:convergence}}

\begin{proof} The objective of the proof is to provide an upper bound of the norm of the approximation error $\tilde U_{1,n}$, which depends explicitly on order $n$. To do so, let us first rewrite the expression of this error as follows
\begin{equation*}
\begin{array}{lcl}
    \tilde{U}_{1,n}(\theta) 
    &=&\displaystyle U(h\!+\!\theta)B\!-\! \underset{k=0}{\overset{n-1}{\sum}}\mathbf{U}_{1}^k l_k\!\left(\!\frac{\theta+h}{h}\!\right),
\end{array}
\end{equation*}
with $\mathbf{U}_{1}^k$ being the projection of $U(\theta+h)B$ onto the $k^{th}$ Legendre polynomial, that is
\begin{align}\label{eq:ak1}
    \mathbf{U}_{1}^k &= \frac{2k+1}{h}\int_{-h}^0 U(\theta+h)B l_k\!\left(\!\frac{\theta+h}{h}\!\right)\dth,\nonumber\\
    &= (2k+1)\int_0^1 U(h\theta) l_k(\theta)\dth.
\end{align}
Using the differentiation rule~\eqref{eq:derlk}, the previous expression can be rewritten as 
\begin{equation}\label{eq:process1}
    \mathbf{U}_{1}^k = \frac{1}{2} \int_0^1 U (h\theta)B\! \left(l_{k+1}^\prime(\theta)-l_{k-1}^\prime(\theta)\right) \dth.
\end{equation}
Then, an integration by parts yields
\begin{equation}\label{eq:process2}
    \mathbf{U}_{1}^k = \frac{h}{2} \int_0^1 U^\prime(h\theta)B \left(l_{k-1}(\theta)-l_{k+1}(\theta)\right) \dth,
\end{equation}
where we have used $l_{k+1}(0)=l_{k-1}(0)$ and $l_{k+1}(1)=l_{k-1}(1)$ (see~\eqref{eq:legeval}), which cancels the first terms of the integration by parts. Repeating this operation, we get
\begin{equation*}
    \begin{array}{lcl}
    \mathbf{U}_{1}^k\!  &=& \displaystyle \frac{h^2}{4(2k\!-\!1)} \int_0^1 U^{\prime\prime}(h\theta)B\left(l_{k-2}(\theta)-l_k(\theta)\right)\dth\\
    &&-\displaystyle \frac{h^2}{4(2k+3)}\int_0^1 U^{\prime\prime}(h\theta)B\left(l_{k}(\theta)-l_{k+2}(\theta)\right)\dth.\\
    \end{array}
\end{equation*}
Then, 
Property \eqref{eq:boundlk2} ensures that
\begin{equation*}
    \norm{\mathbf{U}_{1}^k}  \leq\displaystyle   \frac{\sqrt{\frac{\pi}{2}}h^2\norm{B}}{2\sqrt{k-2}(2k-1)}\int_0^1 \frac{\norm{U^{\prime\prime}(h\theta)}}{\sqrt{\theta(1-\theta)}}\dth.
\end{equation*}
Using \eqref{eq:Uregul1}, $(2k\!-\!1)\geq 2(k-2)$ and $\int_0^1 \frac{\dth}{\sqrt{\theta(1-\theta)}} = \pi$, the following upper bound is obtained
\begin{equation*}
    \norm{\mathbf{U}_{1}^k} \leq\displaystyle  
   \frac{\rho (\frac{\pi}{2})^\frac{3}{2}h^2\norm{B}\norm{W}}{2(k-2)^\frac{3}{2}}, \quad \forall k\geq 3.
\end{equation*}
Applying now \eqref{eq:boundlk} ensures that for any integer $N\geq n$ and for all $\theta$ in $[-h,0]$
\begin{equation*}
     \norm{\underset{k=n}{\overset{N}{\sum}} \mathbf{U}_{1}^k l_k\!\left(\!\frac{h\!+\!\theta}{h}\!\right)}
    \leq \underset{k=n}{\overset{N}{\sum}} \norm{\mathbf{U}_{1}^k}\leq  \sum_{k=n}^N \frac{\varrho_1 \norm{W}}{2(k-2)^\frac{3}{2}},
\end{equation*}
denoting $\varrho_1=\rho(\frac{\pi}{2})^\frac{3}{2}h^2\norm{B}$. Finally, using an integral over estimation of the sum, we obtain 
\begin{equation*}
\norm{\underset{k=n}{\overset{N}{\sum}} \mathbf{U}_{1}^k l_k\!\left(\!\frac{\theta\!+\!h}{h}\!\right)} \leq \displaystyle \varrho_1 \norm{W} \left(\frac{1}{\sqrt{n-3}}- \frac{1}{\sqrt{N-3}}\right).
\end{equation*}
We conclude that the approximation error $\tilde{U}_{1,n}(\theta)=\underset{k=n}{\overset{\infty}{\sum}} \mathbf{U}_{1}^k l_k\!\left(\frac{\theta+h}{h}\right)$ is uniformly bounded as in~\eqref{eq:Wang} and converges also uniformly to zero as $n$ tends to infinity.
\end{proof}

\subsection{Proof of item (ii) of Lemma~\ref{lem:convergence}}

\begin{proof} The objective of the proof is to demonstrate the uniform convergence towards zero of
\begin{equation*}
\begin{array}{lcl}
    \tilde{U}_{1,n}^\prime(\theta) &=&\displaystyle \frac{\mathrm{d}}{\dth} \left(U(\theta+h)B \!-\! \mathbf{U}_{1,n} \ell_n\!\left(\!\frac{\theta\!+\!h}{h}\!\right) \right),\\
    &=&\displaystyle U^\prime(\theta+h)B\!-\! \frac{1}{h} \underset{k=0}{\overset{n-1}{\sum}}\mathbf{U}_{1}^k l_k^\prime\!\left(\!\frac{\theta\!+\!h}{h}\!\right),
\end{array}
\end{equation*}
with $\mathbf{U}_{1}^k$ given by~\eqref{eq:ak1}.
Here, we repeat the process~\eqref{eq:process1}-\eqref{eq:process2} four times successively to obtain
\begin{equation*}
    \begin{array}{lcl}
    \mathbf{U}_{1}^k \!\! &=& \!\! \displaystyle \frac{h^4}{2^4(2k\!-\!5)(2k\!-\!3)(2k\!-\!1)} \int_0^1\!\! U^{(4)}(h\theta)Bl_{k-4}(\theta)\dth\\
    &-&\!\! \displaystyle \frac{h^4}{2^4(2k\!-\!5)(2k\!-\!3)(2k\!-\!1)} \int_0^1\!\! U^{(4)}(h\theta)Bl_{k-2}(\theta)\dth\\
    &-&\!\! \displaystyle \frac{h^4}{2^4(2k\!-\!3)(2k\!-\!1)^2} \int_0^1\!\! U^{(4)}(h\theta)Bl_{k-2}(\theta)\dth\\
    &+&\!\! \displaystyle \frac{h^4}{2^4(2k\!-\!3)(2k\!-\!1)^2} \int_0^1\!\! U^{(4)}(h\theta)Bl_k(\theta)\dth\\
    &-&\!\! \displaystyle \frac{h^4}{2^4(2k\!-\!1)^2(2k\!+\!1)} \int_0^1\!\! U^{(4)}(h\theta)Bl_{k-2}(\theta)\dth\\
    &+&\!\! \displaystyle \frac{h^4}{2^4(2k\!-\!1)^2(2k\!+\!1)} \int_0^1\!\! U^{(4)}(h\theta)Bl_k(\theta)\dth\\
    &-&\!\! \displaystyle \frac{h^4}{2^4(2k\!-\!1)(2k\!+\!1)(2k\!+\!3)} \int_0^1\!\! U^{(4)}(h\theta)Bl_{k-2}(\theta)\dth\\
    &+&\!\! \displaystyle \frac{h^4}{2^4(2k\!-\!1)(2k\!+\!1)(2k\!+\!3)} \int_0^1\!\! U^{(4)}(h\theta)Bl_k(\theta)\dth\\
    &+&\!\! \displaystyle \frac{h^4}{2^4(2k\!-\!1)(2k\!+\!1)(2k\!+\!3)} \int_0^1\!\! U^{(4)}(h\theta)Bl_k(\theta)\dth\\
    &-&\!\! \displaystyle \frac{h^4}{2^4(2k\!-\!1)(2k\!+\!1)(2k\!+\!3)} \int_0^1\!\! U^{(4)}(h\theta)Bl_{k+2}(\theta)\dth\\
    &+&\!\! \displaystyle \frac{h^4}{2^4(2k\!+\!1)(2k\!+\!3)^2} \int_0^1\!\! U^{(4)}(h\theta)Bl_k(\theta)\dth\\
    &-&\!\! \displaystyle \frac{h^4}{2^4(2k\!+\!1)(2k\!+\!3)^2} \int_0^1\!\! U^{(4)}(h\theta)Bl_{k+2}(\theta)\dth\\
    &+&\!\! \displaystyle \frac{h^4}{2^4(2k\!+\!3)^2(2k\!+\!5)} \int_0^1\!\! U^{(4)}(h\theta)Bl_k(\theta)\dth\\
    &-&\!\! \displaystyle \frac{h^4}{2^4(2k\!+\!3)^2(2k\!+\!5)} \int_0^1\!\! U^{(4)}(h\theta)Bl_{k+2}(\theta)\dth\\
    &-&\!\! \displaystyle \frac{h^4}{2^4(2k\!+\!3)(2k\!+\!5)(2k\!+\!7)} \int_0^1\!\! U^{(4)}(h\theta)Bl_{k+2}(\theta)\dth\\
    &+&\!\! \displaystyle \frac{h^4}{2^4(2k\!+\!3)(2k\!+\!5)(2k\!+\!7)} \int_0^1\!\! U^{(4)}(h\theta)Bl_{k+4}(\theta)\dth,\\
    \!\!&=& \!\! \displaystyle \left(\frac{h}{2}\right)^{\!4}\!\frac{\displaystyle \int_0^1\!\! U^{(4)}(h\theta)B\left(\sum_{i=0}^4\binom{4}{i}\alpha_{k,i}l_{k-4+2i}(\theta)\right)\dth}{
    (2k\!-\!5)(2k\!-\!3)(2k\!-\!1)},
    \end{array}
\end{equation*}
where $\alpha_{k,i}$ are positive coefficients whose expression is omitted for simplicity but which verify $\norm{\alpha_{k,i}}\leq 1$.
Therefore, an upper bound of the norm of $\mathbf{U}_{1}^k$ can be derived using property \eqref{eq:boundlk2} on the $2^4$ terms, yielding 
\begin{equation*}
    \norm{\mathbf{U}_{1}^k}  \leq\displaystyle   \frac{\frac{1}{2}\sqrt{\frac{\pi}{2}}h^4\norm{B}}{(2k\!-\!5)(2k\!-\!3)(2k\!-\!1)}\!\int_0^1\!\! \frac{\norm{U^{(4)}(h\theta)}}{\sqrt{(k\!-\!4)\theta(1-\theta)}}\dth.
\end{equation*}
Using \eqref{eq:Uregul2}, $(2k-5)\leq 2(k-4)$ and $\int_0^1 \frac{\dth}{\sqrt{\theta(1-\theta)}} = \pi$, the following upper bound is obtained
\begin{equation}\label{eq:boundUk}
    \norm{\mathbf{U}_{1}^k} \leq\displaystyle  
   \frac{\rho' (\frac{\pi}{2})^\frac{3}{2}h^4\norm{B}\norm{W}}{2(k-4)^\frac{3}{2}(2k\!-\!3)(2k\!-\!1)}, \quad \forall k\geq 5.
\end{equation}
We denote $\varrho_2=\frac{1}{2}\rho'(\frac{\pi}{2})^\frac{3}{2}h^3\norm{B}$ and we use this upper bound to obtain the result. The application of~\eqref{eq:boundderlk} and~\eqref{eq:boundUk} gives, for any integer $N\geq n \geq 5$ and for all $\theta$ in $[-h,0]$, 
\begin{equation*}
\begin{aligned}
    \norm{\frac{1}{h} \underset{k=n}{\overset{N}{\sum}} \mathbf{U}_{1}^k l_k^\prime\!\left(\!\frac{\theta\!+\!h}{h}\!\right)}
    &\leq \underset{k=n}{\overset{N}{\sum}} \frac{\norm{\mathbf{U}_{1}^k}k(k+1)}{h},\\
    &\leq  \sum_{k=n}^N \frac{\varrho_2\norm{W}k(k+1)}{(k-4)^\frac{3}{2}(2k\!-\!3)(2k\!-\!1)}.
\end{aligned}
\end{equation*}
Noticing that $\frac{k(k+1)}{(2k-1)(2k-3)}<\frac{1}{2}$ for all $k\geq 5$, we obtain
\begin{equation*}
    \norm{\frac{1}{h} \underset{k=n}{\overset{N}{\sum}} \mathbf{U}_{1}^k l_k^\prime\!\left(\!\frac{\theta\!+\!h}{h}\!\right)}
    \leq  \sum_{k=n}^N \frac{\varrho_2\norm{W}}{2(k-4)^\frac{3}{2}}.
\end{equation*}
Finally, an integral over estimation of the sum leads to
\begin{equation*}
\norm{\frac{1}{h} \underset{k=n}{\overset{N}{\sum}} \mathbf{U}_{1}^k l_k^\prime\!\left(\!\frac{\theta\!+\!h}{h}\!\!\right)} \!\leq\! \displaystyle \varrho_2\norm{W}\! \left(\!\frac{1}{\sqrt{n\!-\!5}}\!-\! \frac{1}{\sqrt{N\!-\!5}}\!\right)\!.
\end{equation*}
We conclude that the approximation error $\tilde{U}_{1,n}'(\theta)\!=\! \displaystyle\frac{\mathrm{d}}{\dth}\underset{k=n}{\overset{\infty}{\sum}} \mathbf{U}_{1}^k l_k\!\left(\frac{\theta+h}{h}\right)$ is uniformly bounded as in \eqref{eq:Wang2} and converges to zero as $n$ tends to infinity.
\end{proof}

\vspace{-0.2cm}
\subsection{Proof of item (iii) of Lemma~\ref{lem:convergence}}
\vspace{-0.2cm}

\begin{proof} The objective of the proof is to demonstrate the uniform convergence towards zero of
\begin{equation*}
\begin{array}{lcl}
    \tilde{U}_{2,n}(\theta) &=&\displaystyle B^\top U(\theta)B \!-\! \mathbf{U}_{2,n} \ell_n\!\left(\!\frac{\theta\!+\!h}{2h}\!\right),\\
    &=&\displaystyle B^\top U(\theta)B\!-\! \underset{k=0}{\overset{n-1}{\sum}}\mathbf{U}_{2}^k l_k\!\left(\!\frac{\theta\!+\!h}{h}\!\right),
\end{array}
\end{equation*}
with $\mathbf{U}_{2}^k$ the $k$-th coefficient of matrix $\mathbf{U}_{2,n}$ given by
\begin{align}\label{eq:ak2}
    \mathbf{U}_{2}^k &= \frac{2k+1}{2h}\int_{-h}^h B^\top U(\theta)B l_k\!\left(\!\frac{\theta\!+\!h}{2h}\!\right)\dth,\nonumber\\
    &= (2k+1)\int_0^1 B^\top U\!\big(h(2\theta-1)\big)B l_k(\theta)\dth.
\end{align}
Using the differentiation rule~\eqref{eq:process1} and integration by parts~\eqref{eq:process2} on both intervals $(0,\frac{1}{2})$ and $(\frac{1}{2},1)$, we obtain
\begin{equation*}
    \mathbf{U}_{2}^k = h \int_0^1 B^\top U^\prime\!\big(h(2\theta-1)\big)B\! \left(l_{k-1}(\theta)-l_{k+1}(\theta)\right) \dth,
\end{equation*}
since $l_{k+1}(\theta)=l_{k-1}(\theta)$ for $\theta\in\{0,1\}$ is ensured by~\eqref{eq:legeval} and since the continuity of $U\!\big(h(2\theta-1)\big)$ at $\theta=\frac{1}{2}$ is ensured by Property~\ref{prop:DeltaU}~(i). 
Then, repeating this operation, we get to
\begin{equation*}
    \begin{aligned}
    \mathbf{U}_{2}^k\!  &= - \displaystyle \frac{h}{2(2k\!-\!1)}B^\top \Delta U^\prime(0) B \left(l_{k}(1/2)-l_{k-2}(1/2)\right)\\
    & +\! \displaystyle \frac{h^2}{(2k\!-\!1)} \!\int_0^1\!\! B^\top U^{\prime\prime}\!\big(h(2\theta\!-\!1)\big)B\left(l_{k-2}(\theta)\!-\!l_k(\theta)\right)\!\dth\\
    &+ \displaystyle \frac{h}{2(2k\!+\!3)}B^\top \Delta U^\prime(0) B \left(l_{k+2}(1/2)-l_{k}(1/2)\right)\\
    &-\! \displaystyle \frac{h^2}{(2k\!+\!3)}\! \int_0^1\!\! B^\top U^{\prime\prime}\!\big(h(2\theta\!-\!1)\big)B\left(l_{k}(\theta)\!-\!l_{k+2}(\theta)\right)\!\dth\!.
\end{aligned}
\end{equation*}
where $\Delta U^\prime(0):= \lim_{\epsilon\rightarrow 0} \left( U'(\epsilon) - U'(-\epsilon) \right)$. 
Then, an upper bound of the norm of $\mathbf{U}_{2}^k$ can be derived by the use of property \eqref{eq:boundlk2} for $\theta\in[0,1]$ and especially for $\theta=\frac{1}{2}$, yielding 
\begin{equation*}
    \norm{\mathbf{U}_{2}^k}  \!\leq\! \displaystyle \frac{\sqrt{2\pi}h\norm{B}^2}{\sqrt{k\!-\!2}(2k\!-\!1)}\!\! \left(\!\!\norm{\Delta U^\prime(0)}\!+\!h\!\!\int_0^1\!\! \frac{\norm{U^{\prime\prime}\!\big(\!h(2\theta\!-\!1)\!\big)\!}}{\sqrt{\theta(1-\theta)}}\!\dth\!\!\right)\!.
\end{equation*}
Thanks to Property~\ref{prop:DeltaU}~(ii), upper bound \eqref{eq:Uregul1}, inequality $(2k\!-\!1)\leq 2(k\!-\!2)$ and $\int_0^1 \frac{\dth}{\sqrt{\theta(1-\theta)}} = \pi$, we have 
\begin{equation*}
    \norm{\mathbf{U}_{2}^k}  \!\leq\! \displaystyle \frac{\sqrt{2\pi}h\norm{B}^2}{2(k-2)^\frac{3}{2}} (1\!+\!\rho\pi h)\norm{W} = \frac{\varrho_3\norm{W}}{2(k-2)^\frac{3}{2}}, \ \forall k\geq 3 .
\end{equation*}
where notation $\varrho_3:= \sqrt{2\pi} ( 1 + \rho\pi h) h\norm{B}^2$ is introduced. 
For the same reasons stated in the proof of Lemma~\ref{lem:convergence}~(i), for any integer $N\geq n$ and for all $\theta$ in $[0,1]$, the following upper bound is obtained
\begin{equation*}
\norm{\underset{k=n}{\overset{N}{\sum}} \mathbf{U}_{2}^k l_k\!\left(\!\frac{\theta\!+\!h}{h}\!\right)} \leq \displaystyle \varrho_3 \norm{W} \left(\frac{1}{\sqrt{n-3}}- \frac{1}{\sqrt{N-3}}\right).
\end{equation*}
We conclude that the series $\tilde{U}_{2,n}(\theta)=\underset{k=n}{\overset{\infty}{\sum}} \mathbf{U}_{2,n}^k l_k\!\left(\frac{\theta+h}{h}\right)$ exists, converges to zero as $n$ tends to infinity and that inequality~\eqref{eq:Wangbar} holds.
\end{proof}

\bibliographystyle{plain}        
\bibliography{autosam}

\end{document}